\documentclass{amsart}
\usepackage{setspace}

\usepackage{amsmath}                            
\usepackage{amssymb}                            
\usepackage{amsthm}                             
\usepackage{graphicx}                           
\usepackage[pdfborder={0 0 0}]{hyperref}        
\usepackage{comment}                            
\usepackage{blkarray}                           
\usepackage{bm}                                 
\usepackage{multirow}                           
\usepackage{fullpage}                           
\usepackage{xcolor}                             
\usepackage{breqn}                              

\newcommand{\BF}[1]{\mathbf{#1}}

\newcommand{\pos}{\operatorname{pos}_{\mathbb{R}}}
\newcommand{\rank}{\operatorname{rank}}

\newcommand{\BFG}[1]{\bm{#1}}
\newcommand{\semigroup}[1]{\operatorname{pos}_{\mathbb{N}}(#1)}
\newcommand{\lrfan}[1]
{\mathcal{L}\mathcal{R}_{#1}}
\newcommand{\lrcone}[1]{\textsf{H}^+_{#1}}
\newcommand{\lattice}[1]{\mathcal{L}{#1}}

\newtheorem{theo}{Theorem}[section]
\newtheorem{cor}[theo]{Corollary}
\newtheorem{lem}[theo]{Lemma}
\newtheorem{prop}[theo]{Proposition}
\newtheorem{que}[theo]{Question}
\newtheorem{conj}[theo]{Conjecture}

\theoremstyle{definition}
\newtheorem{defn}[theo]{Definition}
\newtheorem{eg}[theo]{Example}

\theoremstyle{remark}
\newtheorem{rem}[theo]{Remark}

\title{External columns and chambers of vector partition functions}
\date{\today}
\author{Stefan Trandafir}
\address{S. Trandafir \\ Department of Mathematics \\
Simon Fraser University\\
8888 University Dr W Burnaby, Canada}

\begin{document}

\begin{abstract}
    The vector partition function $p_A$ associated to a $d \times n$ matrix $A$ with integer entries is the function $\mathbb{Z}^d \to \mathbb{N}$ defined by $\BF{b} \to \#\{\BF{x} \in \mathbb{N}^n : A\BF{x} = \BF{b}\}$. It is known that vector partition functions are piecewise quasi-polynomials whose domains of quasi-polynomiality are maximal cones (chambers) of a fan called the chamber complex of $A$.
    In this article we introduce \emph{external columns} and \emph{external chambers} of vector partition functions. Our main result is that (up to a saturation condition) the quasi-polynomial associated to a chamber containing external columns arises from a vector partition function with $k$ fewer equations and variables. In the case that the chamber is external -- that is, when the number of external columns in a chamber is as large as possible without being trivial -- the quasi-polynomial arises from a coin exchange problem. By exploiting this we are able to obtain a determinantal formula, characterize when the quasi-polynomial is polynomial, and show that in this case it is actually given by a negative binomial coefficient. We then apply these results to the enumeration of loopless multigraphs satisfying some degree conditions. Finally, we suggest a generalization to a result of Baldoni and Vergne for polynomials arising from chambers that we call \emph{semi-external chambers}. 
\end{abstract}

\maketitle

\section{Introduction} \label{sec:introduction}
The vector partition enumeration problem is the natural $d$-dimensional analogue of the well-studied coin exchange problem (see for example \cite[Chapter 1]{BeRo15}). More precisely, for some full rank $d \times n$ matrix $A$ with integer entries, the number of vector partitions of $\BF{b} \in \mathbb{Z}^d$ is the number of solutions $\BF{x} \in \mathbb{N}^n$ of the equation $A\BF{x} = \BF{b}$. If this number is finite for all $\BF{b} \in \mathbb{Z}^d$, we say that the function $p_A : \mathbb{Z}^d \to \mathbb{N}$ yielding the number of vector partitions of $\BF{b}$ for all $\BF{b} \in \mathbb{Z}^d$ is the vector partition function associated to the matrix $A$.

Vector partition functions appear in many different contexts with applications in statistics, representation theory, algebraic geometry, and more. More concretely, the enumeration of Contingency Tables \cite{DeSt03}, the computation of Littlewood-Richardson coefficients \cite{Ra04}, and of Kronecker coefficients \cite{MiRoSu21}, constitute some examples where vector partition functions play a major role.

It is known that the vector partition function is a piecewise quasi-polynomial whose domains of quasi-polynomiality are the maximal cones (chambers) of a fan \cite{Stur94}. In this article we introduce the notion of \emph{external columns} of $A$ and \emph{external chambers} of $A$. Our main result (Theorem \ref{theo:external-col-variables}) is that (up to a saturation condition) one may compute the quasi-polynomial for a chamber $\gamma$ (henceforth denoted $p_A^{\gamma}$) containing external columns by computing a quasi-polynomial for a corresponding chamber in a simpler vector partition function. More precisely, if there are $k$ external columns of $A$ in $\gamma$, then the quasi-polynomial arises from a vector partition function $p_B$ for some matrix $B$ with $k$ fewer rows and columns than $A$. 

In the case that $\gamma$ is an external chamber (contains the maximal amount of external columns without being trivial), Theorem \ref{theo:external-col-variables} implies that the quasi-polynomial associated to $\gamma$ arises from a coin exchange problem. Through this observation we are able characterize when the quasi-polynomial associated to $\gamma$ is actually polynomial. Moreover, in these instances the formulae are particularly simple: they are given by negative binomial coefficients (Theorem \ref{theo:external-chamber-binom}). Unimodular matrices satisfy the required saturation conditions, and so we immediately obtain results for two notions of unimodularity (Corollaries \ref{cor:det-unimodular} and \ref{cor:unimodular}). 

By identifying external chambers in a family of vector partition fuctions arising from the enumeration of multigraphs, we are able to obtain explicit counting formulae for loopeless multigraphs satisfying certain degree requirements (Theorem \ref{theo:multigraph}). 

We also general set of chambers (of which the external chambers take part) that we call \emph{semi-external chambers} and suggest a generalization of a result of Baldoni \& Vergne (Conjecture \ref{conj:lin-factors}) related to linear factors of polynomials of $p_A$. This could potentially be used to better understand the vector partition functions associated to the enumeration of multigraphs, as well as in the analysis of certain polynomials associated to the Littlewood-Richardson coefficients. 

The article is organized as follows: in Section \ref{sec:geometry} we recall background on cones, fans, and vector partition functions. In Section \ref{sec:external-cols-chambers}, we introduce external columns and chambers and derive some useful properties associated to them. Next, in Section \ref{sec:dim-reduction}, we derive our main result: Theorem \ref{theo:external-col-variables}. In Section \ref{sec:external-chamber-case} we consider the case of external chambers in more detail in order to derive a determinantal formula (Theorem \ref{theo:ehrhart-det}), as well as to characterize when the quasi-polynomial associated to such a chamber is actually polynomial, in which case we show that it is actually given by a negative binomial coefficient (Theorem \ref{theo:external-chamber-binom}). We further show that this formula can be simplified if the given matrix is unimodular (Corollary \ref{cor:unimodular}). In Section \ref{sec:multigraph} we give an application of our work to the enumeration of multigraphs. In Section \ref{sec:lem-proof} we prove the main ingredient of Theorem \ref{theo:external-col-variables} -- a technical result: Lemma \ref{lem:dim-reduction}. In Section \ref{sec:lin-factors} we consider a more general set of chambers (of which the external chambers take part) that we call \emph{semi-external chambers}. Finally, in Section \ref{sec:future} we discuss possible generalizations and applications of our work. 

\section{Background} \label{sec:geometry}

\subsection{Cones} \label{subsec:cones}
We mostly follow Cox-Little-Schenck \cite{CoLiSc11} and Fulton \cite{Fu93} in this section. If the reader is familiar with cones, we advise them to read Remark \ref{rem:notation} before skipping to the following section.

A \emph{convex polyhedral cone} in $\mathbb{R}^d$ is a set of the form 
\[
\sigma = \pos(\BF{u}_1, \dots, \BF{u}_k) = \{\lambda_1\BF{u}_1 + \lambda_2\BF{u}_2 + \dots + \lambda_k\BF{u}_k : \lambda_1, \lambda_2, \dots, \lambda_k \geq 0\}
\]
and we say that $\sigma$ is \emph{generated} by $\BF{u}_1, \dots, \BF{u}_k$, and that $\BF{u}_1, \dots, \BF{u}_k$ are \emph{generators} of $\sigma$. 
We also define $\pos(\emptyset) = \{\BF{0}\}$. The \emph{dimension} of cone $\sigma$ is the dimension of the subspace generated by $\{\BF{u}_1, \dots, \BF{u}_k\}$.  
Let $\sigma$ be a convex polyhedral cone in $\mathbb{R}^d$. We say that the set
\[
\sigma^{\vee} := \{ \BF{m} \in \mathbb{R}^d : \BF{m} \cdot \BF{u} \geq 0 \hbox{ for all } \BF{u} \in \sigma\}
\]
is the \emph{dual cone of $\sigma$}. Note that although it is not immediately obvious from the defintion, $\sigma^{\vee}$ is also a convex polyhedral cone. A convex polyhedral cone $\sigma$ is a \emph{pointed convex polyhedral cone} (or \emph{strongly convex polyhedral cone}) if it contains no $1$-dimensional subspace of $\mathbb{R}^d$, and $\sigma$ is \emph{rational} if it can be generated by a finite number of integer points. In general in this article, we deal with rational strongly convex polyhedral cones, so we simply call these \emph{cones} for short.

For $\BF{m} \in \mathbb{R}^d$, $\BF{m} \neq \BF{0}$, let
\[
H_{\BF{m}} := \{\BF{u} \in \mathbb{R}^d : \BF{m} \cdot \BF{u} = 0\}
\]
and 
\[
H_{\BF{m}}^+ := \{\BF{u} \in \mathbb{R}^d : \BF{m} \cdot \BF{u} \geq 0\}.
\]
We say that $H_{\BF{m}}$ is a \emph{supporting hyperplane} of a cone $\sigma$ if $\sigma \subseteq H_{\BF{m}}^+$, and in this case we also say that $H_{\BF{m}}^+$ is a \emph{supporting half-space} of $\sigma$. A \emph{face} of cone $\sigma$ is a set of the form $H_{\BF{m}} \cap \sigma$. A \emph{ray} (or \emph{edge}) of $\sigma$ is a face of dimension 1 and a \emph{facet} of $\sigma$ is a face of codimension 1. We also use the term ray on its own to refer to any $1$-dimensional cone. Any proper face of $\sigma$ is the intersection of all the facets containing it.

Given a ray $r$, we say that $\BF{w} = (w_1, \dots, w_d)$ is a \emph{ray generator} of $r$ if $r = \pos(\BF{w})$. If additionally, $\BF{w} \in \mathbb{Z}^d$ and $\gcd(\{w_i : 1 \leq i \leq d\}) = 1$, then we say that $\BF{w}$ is the minimal ray generator of $r$. Any set of generators of a cone $\sigma$ contains some minimal subset of generators, which still generate $\sigma$. This subset consists of exactly the vectors which generate the edges of $\sigma$. We call such a subset a \emph{minimal generating set} of $\sigma$, and call its elements \emph{ray generators of $\sigma$}. If a minimal generating set $S$ of $\sigma$ additionally has the property that each element is a minimal ray generator, then we say that $S$ is the set of \emph{minimal ray generators} of $\sigma$. Each cone $\sigma$ has a unique set of minimal ray generators. We say that $\sigma$ is \emph{simplicial} if its set of ray generators is linearly independent.

\begin{rem} \label{rem:notation}
    We remark here that we have deviated slightly from the notation of \cite{CoLiSc11} and \cite{Fu93}. In \cite{Fu93} what we call ``ray generators'' are called ``minimal generators''. In \cite{CoLiSc11} what we refer to as the ``minimal ray generator'' is simply called the ``ray generator'', and what we call ``ray generators'' of $\sigma$ are called ``minimal generators'' of $\sigma$.  
    Our choice is dictated by our introduction of the terms ``external ray generators'' and ``minimal external ray generators'' for which either notation scheme (i.e that of \cite{CoLiSc11} or \cite{Fu93})  would cause confusion. 
\end{rem}

For a cone $\sigma$ with supporting hyperplane $H$, we say that $\BF{u}$ is an \emph{inner facet normal} of $\sigma$ if $\BF{u} \in H^+$ is normal to $H$, and that $\BF{u}$ is an \emph{outer facet normal} of $\sigma$ if $\BF{u} \in H^-$ is normal to $H$. The following result is given in \cite[Proposition 1.2.8]{CoLiSc11} and the discussion appearing immediately after Proposition 1.2.8.

\begin{prop}[Cox-Little-Schenck, Proposition 1.2.8] \label{prop:facet-normal-dual-ray}
	Let $\sigma \subseteq \mathbb{R}^d$ be a polyhedral cone such that $\sigma = H_{\BF{m}_1}^+ \cap  H_{\BF{m}_2}^+ \cap \dots \cap H_{\BF{m}_s}^+$. Then $\sigma^{\vee} = \pos(\BF{m}_1, \BF{m}_2, \dots,\BF{m}_s)$. In particular, $\BF{m}_1, \BF{m}_2, \dots, \BF{m}_s$ are the inner facet normals of $\sigma$ if and only if $\BF{m}_1, \BF{m}_2, \dots, \BF{m}_s$ generate the rays of $\sigma^{\vee}$.
\end{prop}

A \emph{fan} $\Sigma$ is a set of cones such that if $\sigma \in \Sigma$, then every face of $\sigma$ is in $\Sigma$, and the intersection of $\sigma_1, \sigma_2 \in \Sigma$ is a face of both $\sigma_1$ and $\sigma_2$. Maximal cones of $\Sigma$ are called \emph{chambers}.

\subsection{Vector partition functions} \label{subsec:vpf}

Here we give a brief primer on vector partition functions, generally following the notation of \cite{Stur94}. My $\mathbb{N}$ we denote the set of non-negative integers $\{0, 1, 2, \dots, \}$. Throughout this document $A$ will denote a $ d \times n$ matrix of rank $d$ with integer entries and 
\begin{equation*}
\ker(A) \cap \mathbb{R}^n_{\geq 0} = \{\BF{0}\}.
\end{equation*}
We begin by recalling the definition of the vector partition function.
\begin{defn}
The \emph{vector partition function of $A$}
\begin{equation*}
p_A : \mathbb{Z}^d \to \mathbb{N}
\end{equation*}
is defined by
\begin{equation*}
    p_A(\BF{b}) := \#\{\BF{x} \in \mathbb{N}^n : A\BF{x} = \BF{b}\}.
\end{equation*}
\end{defn}

One can view this problem as enumerating the number of ``partitions'' of the vector $\BF{b}$ whose parts are the columns of $A$ -- hence the name vector partition function.

\begin{rem}
    Recall that the condition $\ker(A) \cap \mathbb{R}^n_{\geq 0} = \{\BF{0}\}$ is imposed so that $p_A$ is indeed a function. Otherwise we may have that $p_A(\BF{b})$ is not finite for some $\BF{b} \in \mathbb{Z}^d$.
\end{rem}

Given a matrix $M$ with columns $\BF{m}_1, \BF{m}_2, \dots, \BF{m}_k$, we define the \emph{cone associated to $M$}, denoted $\pos(M)$ to be the set $\{\lambda_1\BF{m}_1 + \lambda_2\BF{m}_2 + \dots + \lambda_k\BF{m}_k : \lambda_1,\dots,\lambda_k \geq 0\}$. In words, $\pos(M)$ is the cone generated by the columns of $M$.\footnote{Some authors use $\pos(M)$ to indicate the cone generated by the rows of $M$.}
For each $s \subseteq \{1, 2, \dots, n\}$, we define $A_{s}$ to be the submatrix of $A$ composed of the columns of $s$. For each $s$ satisfying $|s|~=~\rank(A_s)~=~\rank(A)~=~d$, we say that $\pos(A_s)$ is a \emph{simplicial cone of~$A$}. The \emph{chamber complex} of $A$ is the fan obtained as the common refinement of the simplicial cones of $A$ (viewed as fans). 
Explicitly, the chamber complex is the set of cones
\begin{equation*}
    \left\{\bigcap_{\substack{s \subseteq \{1, 2, \dots, n\} \\ |s|~=~\rank(A_s)~=~d \\ \mathbf{b} \in \pos(A_s)}} \pos(A_s) :  \mathbf{b} \in \pos(A)\right\}
\end{equation*}
along with all of their faces. In words, for each $\BF{b} \in \pos(A)$, form the cone obtained by taking the intersection of all simplicial cones containing $\BF{b}$. The chamber complex is exactly the set of cones obtained along with their faces.

The cones of maximal dimension of the chamber complex are called \emph{geometrical chambers} -- we call them \emph{chambers} for short\footnote{In the literature chambers are usually defined to be the interiors of what we call chambers (i.e the maximal cells of the chamber complex). We use this terminology since we are, much more often than not, interested in the closed sets and not their interiors.}. Equivalently, these are the cones of dimension $d$ of the chamber complex.

We list some facts about the chamber complex that we exploit throughout this thesis. These facts can be derived directly from the definition of the chamber complex.
\begin{enumerate}
    \item Any chamber is exactly the intersection of all simplicial cones containing it. 
    \item For all $\BF{b} \in \pos(A)$, the intersection of all simplicial cones containing $\BF{b}$ ($\operatorname{cone}(\BF{b})$) is a cone of the chamber complex -- in particular if the intersection is $d$-dimensional, it is a chamber.
    \item For a given chamber $\gamma$, if $\BF{b} \in \gamma^{\circ}$ (where $\gamma^{\circ}$ denotes the interior of $\gamma$), then $\gamma$ is the intersection of all simplicial cones containing $\BF{b}$.
    \item If a chamber $\gamma$ of $A$ intersects a simplicial cone $\sigma$ of $A$ $d$-dimensionally, then $\gamma \subseteq \sigma$.
    \item Let $j \in \{1, \dots, n\}$. If column $\BF{a}_j$ is in a chamber $\gamma$, then $\BF{a}_j$ is a ray generator of $\gamma.$
\end{enumerate}

The vector partition function $p_A$ can be described explicitly as a piecewise function whose domains are the chambers of $A$. The functions that are valid on the chambers are \emph{quasi-polynomials}, which are finite sums of the form 
\begin{equation*}
    q(z_1, \dots, z_k) = \sum_{(i_1, \dots, i_k) \in S} c_{i_1, \dots, i_k}(z_1, \dots, z_k)z_1^{i_1}, \dots, z_k^{i_k}
\end{equation*}
where $S \subset \mathbb{Z}^k$ is finite and the $c_{i_1, \dots, i_k}$ are non-zero periodic functions in $(z_1, \dots, z_k)$. That is, there exist positive integers $n_1, \dots, n_k$ such that $c_{i_1, \dots, i_k}(z_1, \dots, z_k) = c_{j_1, \dots, j_k}(z'_1, \dots, z'_k)$ whenever $z_{\ell} \equiv z'_{\ell} \mod n_{\ell}$ for $\ell = 1, \dots, k$. 
The \emph{degree} of the quasi-polynomial $q$ is the maximum over the sums $i_1 + \dots + i_k$, and the \emph{period} of $q$ is the minimal positive integer $N$ such that $c_{i_1, \dots, i_k}(z_1, \dots, z_k) = c_{j_1, \dots, j_k}(z_1, \dots, z_k)$ whenever $i_{\ell} \equiv j_{\ell} \mod N$ for $\ell = 1, \dots, k$ for all $(i_1, \dots, i_k) \in S$. We remark that a quasi-polynomial with period equal to one is just a polynomial.
To indicate the quasi-polynomial associated to a chamber~$\gamma$, we write $p_A^{\gamma}$.

The explicit characterization of the form of the vector partition function is due to Sturmfels, and we record it in the following theorem.

\begin{theo}[Sturmfels, 1994 \cite{Stur94}]
    Let $A$ be a $d \times n$ matrix of rank $d$.
    The vector partition function of $A$, $p_A$, is a piecewise quasi-polynomial of degree $n-d$ whose domains of quasi-polynomiality are the maximal cones (chambers) in the chamber complex of $A$.
\end{theo}
 
For each choice of $\BF{b} \in \mathbb{Z}^d$, the evaluation $p_A(\BF{b})$ is the number of integer points in the polytope 
\begin{equation*}
    \mathcal{P} := \{\BF{x} \in \mathbb{R}^d : A\BF{x} = \BF{b}, \BF{x} \geq \BF{0}\}.
\end{equation*}

The funcion $L_{\mathcal{P}}(t)$ that enumerates integer points in the $t^{th}$ dilate of $\mathcal{P}$
\begin{equation*}
    t\mathcal{P} := \{\BF{x} \in \mathbb{R}^d : A\BF{x} = t\BF{b}, \BF{x} \geq \BF{0}\}
\end{equation*}
is a quasi-polynomial, which is called the \emph{Ehrhart quasi-polynomial}.
This is equal to $p_A(t\BF{b})$ (viewed as a function of $t$). Ehrhart quasi-polynomials have many beautiful properties (the book \emph{Computing the continuous discretely} by Beck and Robbins \cite{BeRo15} gives a wonderful overview of this topic). 

Barvinok's algorithm \cite{Barv94} allows one to explicitly compute the piecewise quasi-polynomial $p_A$ in polynomial time for fixed dimension $n$. Multiple implementations of Barvinok's algorithm exist: \emph{Latte} \cite{DeHe04} developed by De Loera, Hemmecke, Tauzer, and Yoshida can be used to (among many other things) compute the quasi-polynomial $p_A(t\BF{b})$ for each $\BF{b} \in \pos(A)$ and is integrated in \emph{Sagemath}; \emph{Barvinok} \cite{KoVeWo08}, developed by Koeppe, Verdoolaege, and Woods can be used to compute the full piecewise quasi-polynomial $p_A$\footnote{The algorithm implemented by Koeppe, Verdoolaege, and Woods is called the \emph{Barvinok-Woods} algorithm. It is based on the original formulation of Barvinok.}. While Barvinok's algorithm is polynomial time for fixed dimension $n$, the problem of computing the vector partition function quickly becomes intractable as the dimension grows. 

\section{External columns and chambers} \label{sec:external-cols-chambers}

In this section, we introduce the main objects of study of this document: external columns, rays, facets, and chambers. Additionally, we describe some properties of these objects that will be necessary for our main result, Theorem \ref{theo:external-col-variables}. We begin by defining some notation which will appear throughout the rest of this article.

\subsection{Notation}
Recall that $A$ denotes a $ d \times n$ matrix with integer entries, of rank $d$, and satisfying 
\begin{equation*}
    \ker(A) \cap \mathbb{R}_{\geq 0}^n = \{\BF{0}\}.
\end{equation*}
We denote the columns of $A$ by $\BF{a}_1, \dots,\BF{a}_n $. Often we informally identify a matrix with its set of columns, so that we may write $\semigroup{\BF{a}_1, \dots, \BF{a}_n}$ in lieu of $\semigroup{A}$. Additionally, we often refer to a \emph{chamber of the matrix} $A$ instead of saying a \emph{chamber in the chamber complex} of $A$ for short.

Given a $d \times n$ matrix $M$, by $M_{\hat{i}, \hat{j}}$ we denote the $(d - 1) \times (n -1)$ submatrix obtained by removing row $i$ and column $j$. We also denote by $M_{\hat{i}, \cdot}$ the $ (d-1) \times n$ matrix obtained by removing row $i$, and by $M_{\cdot, \hat{j}}$ the $d \times (n-1)$ matrix obtained by removing column $j$. Similarly for a vector $\BF{v}$, by $\BF{v}_{\hat{i}}$, we denote $\BF{v}$ with the $i$th coordinate removed. 

\subsection{External columns} \label{subsec:external-cols}

\begin{defn}
    Let $\BF{a}_j$ be the $j${th} column of $A$ for some $j \in \{1, \dots, n\}$. We define $\BF{a}_j$ to be an \emph{external column} of $A$ if $\BF{a}_j \notin \pos(A_{\cdot, \hat{j}})$. 
\end{defn}

If $\BF{a}_j$ is an external column of $A$, then the cone $\pos(A_{\cdot, \hat{j}})$ is a proper subset of $\pos(A)$. 

\begin{prop} \label{prop:external-col-equivalence}
Let $\BF{a}_j$ be a column of $A$ for some $j \in \{1, \dots, n\}$. Then $\BF{a}_j$ is an external column of $A$ if and only if $\BF{a}_j$ is a ray generator of $\pos(A)$ and no other column of $A$ is in the span of $\BF{a}_j$.
\end{prop}

\begin{proof}
Let $\BF{a}_j$ be an external column of $A$. Since $\pos(A_{\cdot, \hat{j}}) \neq \pos(A)$, $\BF{a}_j$ is part of a minimal generating set of $\pos(A)$, and thus is a ray generator of $\pos(A)$. Also, no other column $\BF{c}$ of $A$ is in its span. Otherwise, either $\BF{c}$ is a positive multiple of $\BF{a}_j$ in which case $\BF{a}_j \in \pos(A_{\cdot, \hat{j}})$ or $\BF{c}$ is a negative multiple of $\BF{a}_j$ in which case $\operatorname{ker}(A) \cap \mathbb{R}^d_{\geq 0} \neq \{\BF{0}\}$. 

Conversely, let $\BF{a}_j$ be a ray generator of $\pos(A)$ with no other column of $A$ in its span. Then $\BF{a}_j$ is in a minimal generating set of $\pos(A)$, and since no other column of $A$ is in its span, $\BF{a}_j \notin \pos(A_{\cdot, \hat{j}})$. Therefore, $\BF{a}_j$ is an external column of $A$.
\end{proof}

A straightforward consequence of this proposition is that external columns of $A$ lie on facets of $\pos(A)$. The following proposition will be useful in the following section, and is also a straightforward consequence of Proposition \ref{prop:external-col-equivalence}.

\begin{prop} \label{prop:external-simplicial}
    Let $\BF{a}_j$ be an external column of $A$ for some $j \in \{1, \dots, n\}$, and let $\pos(A_s)$ for some $s \subseteq \{1, \dots, n\}$. If $\BF{a}_j \in \pos(A_s)$, then $\BF{a}_j$ is a ray generator of $\pos(A_s)$. In particular, $j \in s$.
\end{prop}

One can also define the external columns in terms of the vector partition function $p_A$. They are exactly the columns of $A$ for which $p_A(t\mathbf{a}_j) \leq 1$ for all non-negative integers $t$ (i.e for which the Ehrhart quasi-polynomial $p_A(t\mathbf{a}_j)$ has degree $0$). 

\subsection{External chambers} \label{subsec:external-chambers}
We now introduce external chambers, the main objects of study of our work.

\begin{defn}
Let $\gamma$ be a chamber of $A$. We define $\gamma$ to be an \emph{external chamber} of $A$ if all but one ray of $\gamma$ is generated by an external column of $A$. Further, we define the rays that are generated by external columns of $A$ to be \emph{external rays} of $\gamma$ and the other ray to be the \emph{internal ray} of $\gamma$. Moreover, we define any generator of an external ray to be an \emph{external ray generator} of $\gamma$ and any ray generator of an internal ray to be an \emph{internal ray generator} of $\gamma$.
\end{defn}

\begin{rem} \label{rem:degenerate}
    The reader may wonder what happens in the case that $\gamma$ is generated solely by external columns. This case is degenerate: $A$ has a single chamber and $p_A(\BF{b}) \leq 1$ for all $\BF{b} \in \pos(A) \cap \mathbb{Z}^d$. 
\end{rem}

\begin{eg} \label{eg:easy-eg1}
    Consider the following matrix
    \begin{equation*}
        A^{2,2} = \begin{bmatrix}
            1 & 0 & 1 & 1 \\
            0 & 1 & 1 & 2
        \end{bmatrix}
    \end{equation*}
    given in \cite{MiRoSu21} (where it is named $A_{2,2}$) that is a member of a family of matrices $A^{m,n}$ related to Kronecker coefficients. 
    Call its columns $\BF{a}_1, \BF{a}_2, \BF{a}_3, \BF{a}_4$.
    The chamber complex of $A^{2,2}$ is defined by the three chambers
    \[
    \begin{array}{ccc}
    \gamma_1 = \pos\left(\BF{a}_1, \BF{a}_3 \right), &
    \gamma_2 = \pos\left(\BF{a}_3, \BF{a}_4 \right), &
    \gamma_3  = \pos\left(\BF{a}_2, \BF{a}_3 \right).
    \end{array}
    \]
    The external columns of $T$ are $\mathbf{a}_1$ and $\mathbf{a}_2$, and so we see that $\gamma_1$ and $\gamma_3$ are external chambers while $\gamma_2$ is not. The columns and chambers of $A^{2,2}$ are depicted in Figure \ref{fig:eg1}. 

    \begin{figure}
        \centering
        \includegraphics[scale=0.4]{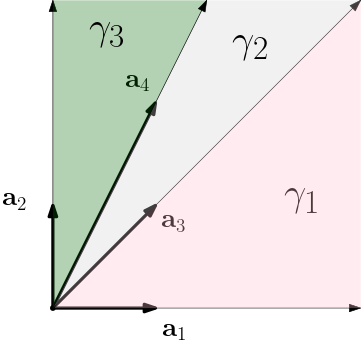}
        \caption{The columns and chambers of $A^{2,2}.$ Columns $\BF{a}_1, \BF{a}_2$ are external while $\BF{a}_3, \BF{a}_4$ are not. Also, chambers $\gamma_1,\gamma_3$ are external while $\gamma_2$ is not.}
        \label{fig:eg1}
    \end{figure}
\end{eg}

It is easy to see from the previous example that when $d=2$ there are at most two external chambers. There may be fewer if there are multiple columns of $A$ generating a ray of $\pos(A)$. 

\begin{eg} \label{eg:kostant-eg1}
    Consider the following matrix
    \begin{equation*}
    K_3 = 
    \begin{bmatrix}
    1 & 1 & 1 & 0 & 0 & 0 \\
    -1 & 0 & 0 & 1 & 1 & 0 \\
    0 & -1 & 0 & -1 & 0 & 1
    \end{bmatrix}.
    \end{equation*}
     This matrix relates to Kostant's partition function for root systems of type $A$ (see \cite[Section 5]{DeSt03} for more details). 
    Call its columns $\BF{a}_1, \BF{a}_2, \dots, \BF{a}_6$, and let 
    \begin{equation*}
        \BF{v} := \begin{bmatrix}
        1 \\
        1 \\
        -1
    \end{bmatrix}.
    \end{equation*}
    The chamber complex of $K_3$ is defined by the seven chambers 
    \[
    \begin{array}{cccc}
    \gamma_1 = \pos\left(\BF{a}_4, \BF{a}_5, \BF{v} \right), &
    \gamma_2 = \pos\left(\BF{a}_3, \BF{a}_5, \BF{v} \right), &
    \gamma_3  = \pos\left(\BF{a}_2, \BF{a}_4, \BF{v} \right), & 
    \gamma_4 = \pos\left(\BF{a}_3, \BF{a}_5, \BF{a}_6 \right), \\
    \gamma_5 = \pos\left(\BF{a}_1, \BF{a}_3, \BF{a}_6 \right), &
    \gamma_6  = \pos\left(\BF{a}_1, \BF{a}_2, \BF{a}_3 \right), &
    \gamma_7  = \pos\left(\BF{a}_2, \BF{a}_3, \BF{v} \right). & \
    \end{array}
    \]
    \begin{figure}
        \centering
        \includegraphics[scale=0.4]{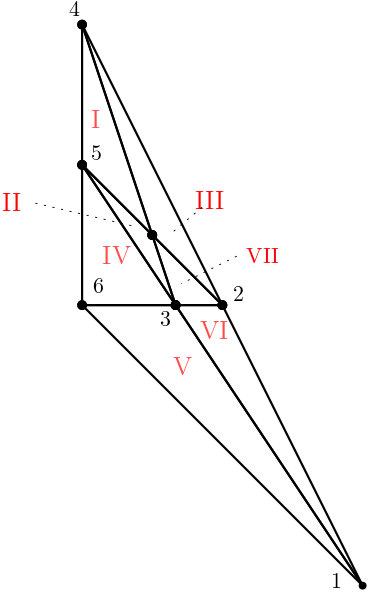}
        \caption{A projection of the chamber complex of $K_3$ via $(x,y,z)~\to~\left(\frac{x}{3x + 2y + z}, \frac{y}{3x + 2y + z}\right)$. Here vertex $i$ is obtained by projecting the $i$th column of $K_3$. The sole unnumbered vertex is the projection of the ray generated by $(1,1,-1)$ obtained in the refinement process. Columns $1,4,6$ are external, and Chamber $V$ is external.}
        \label{fig:kostant-chambers1}
    \end{figure}
A $2$-dimensional projection of the chamber complex of $K_3$ is depicted in Figure \ref{fig:kostant-chambers1}. The projection of the ray generated by each column is the vertex labeled with the appropriate column number. Additionally, the projection of each chamber is given by a labeled (with Roman numerals) 2-dimensional region bounded by edges. One can check (or derive from the figure) that the external columns of $K_3$ are $\mathbf{a}_1, \mathbf{a}_4, \mathbf{a}_6$. Therefore, $\gamma_5$ is an external chamber since its only interal ray generator is $\BF{a}_3$. One may check that each of the other chambers have at least two internal ray generators.
\end{eg}

The next proposition follows directly from the definition of the chamber complex of $A$.

\begin{prop} \label{prop:licolumns-chamber}
If a chamber $\gamma$ of $A$ contains $d$ linearly independent columns of $A$, say $\BF{a}_1, \dots, \BF{a}_d$, then $\gamma = \pos\left(\BF{a}_1, \dots, \BF{a}_d\right)$.
\end{prop}

\begin{prop} \label{prop:chamber-simplicial}
External chambers of $A$ are simplicial.
\end{prop}

\begin{proof} 
Let $\gamma$ be an external chamber of $A$. If $\gamma$ is not simplicial, then it must contain at least $d$ external columns, say $\BF{a}_1, \dots, \BF{a}_{k}$ for some $k \geq d$. By definition, $\gamma$ is contained in some simplicial cone $\pos(A_s)$ for some $s \subseteq \{1, 2, \dots, n\}$ with $|s| = \rank(A_s) = d$. By Proposition~\ref{prop:external-simplicial}, $\{1, \dots, k\}~\subseteq~s$, so $k \leq d$. Thus $k=d$, and we find that
\begin{equation}
    \pos(\BF{a}_1, \dots, \BF{a}_d) \subseteq \gamma \subseteq \pos(A_s) = \pos(\BF{a}_1, \dots, \BF{a}_d).
\end{equation}
Therefore $\gamma = \pos(\BF{a}_1, \dots, \BF{a}_d)$ which is a contradiction since $\gamma$ must have an internal ray generator. 
\end{proof}

The following lemma will prove useful in terms of computing external chambers. Additionally, it plays a key role in the results of Section \ref{sec:dim-reduction}. Figure \ref{fig:unique-facet} provides an illustration of some of the elements of the proof, and may be a useful visual guide. 

\begin{lem} \label{lem:unique-facet}
Let $\BF{a}_j$ be an external column of $A$ for some $j \in \{1, \dots, n\}$. Then any chamber of $A$ containing $\BF{a}_j$ has a single facet $f$ not containing $\BF{a}_j$. Moreover, if $H$ is the supporting hyperplane of $f$, then $H$ separates $\BF{a}_j$ and $\pos(A_{\cdot, \hat{j}}).$
\end{lem}

\begin{proof} 
Since $\BF{a}_j$ is an external column of $A$, by definition $\BF{a}_j \not \in \pos(A_{\cdot, \hat{j}})$. Let $f_1, \dots, f_k$ denote the faces of $\pos(A_{\cdot, \hat{j}})$ of dimension $d-1$. 
We note that these faces are not necessarily facets of $\pos(A_{\cdot, \hat{j}})$: if $\pos(A_{\cdot, \hat{j}})$ is $(d-1)$-dimensional, then $k = 1$ and $f_1 = \pos(A_{\cdot, \hat{j}})$. However, the proof proceeds in the same way regardless of whether $\pos(A_{\cdot, \hat{j}})$ is of dimension $d$ or $d-1$.

For $i=1, \dots, k$, the faces $f_i$ can be described as $f_i = \pos(A_{\cdot, \hat{j}}) \cap H_i$ for some supporting hyperplanes $H_1, \dots, H_k~\subset~\mathbb{R}^d$. Without loss of generality, let $H_1, H_2, \dots, H_{\ell}$ be the set of hyperplanes that separate $\BF{a}_j$ and the cone $\pos(A_{\cdot, \hat{j}})$ and also do not contain $\BF{a}_j$. For $i=1,\dots,\ell$, let $\kappa_i$ be the cone generated by $\BF{a}_j$ and the columns of $A_{\cdot, \hat{j}}$ that generate~$f_i$. Since $f_i$ is generated by columns of $A_{\cdot, \hat{j}}$ spanning $H_i$, it follows that $\kappa_i$ is a union of simplicial cones of $A$ -- that is, $\kappa_i = \cup_{s \in S} \pos(A_{s})$ for some $S \subset \mathcal{P}(\{1, \dots, n\})$ with each $s \in S$ satisfying $|s| = \rank(A_{s}) = d$. 

Let $\gamma$ be a chamber of $A$ containing $\BF{a}_j$. Then $\gamma$ has a $d$-dimensional intersection with a simplicial cone of $A$ that
is contained in $\kappa_{i^*}$ for some $1 \leq i^* \leq \ell$, and so $\gamma \subseteq \kappa_{i^*}$.
Let $\tau$ be a facet of $\gamma$ not contained in~$f_{i^*}$, so $\tau$ is not contained in~$\pos(A_{\cdot, \hat{j}})$. By the definition of the chamber complex, $\tau$ must be contained in a facet $\tau'$ of a simplicial cone $\sigma$ of $A$. Thus, we can write $\tau' = \sigma \cap H'$ for some hyperplane $H' \subseteq \mathbb{R}^d$. Then $\gamma \cap H'$ defines a proper face of $\gamma$ containing the facet $\tau$, and so $\tau = \gamma \cap H'$. Since $\tau'$ is generated by columns of $A$ and not contained in $\pos(A_{\cdot, \hat{j}})$, $\BF{a}_j$ is a ray generator of $\tau'$. Thus, $\BF{a}_j \in H'$, and since $\BF{a}_j \in \gamma$ by assumption, $\BF{a}_j \in \tau$. Finally, we see that the unique facet of $\gamma$ not containing $\BF{a}_j$ is $\gamma \cap H_{i^*} = \gamma \cap f_{i^*}$ and $H_{i^*}$ is a separating hyperplane of $\BF{a}_j$ and $\pos(A_{\cdot, \hat{j}}).$
\end{proof}

\begin{figure}[ht] 
    \centering
    \includegraphics[scale=0.5]{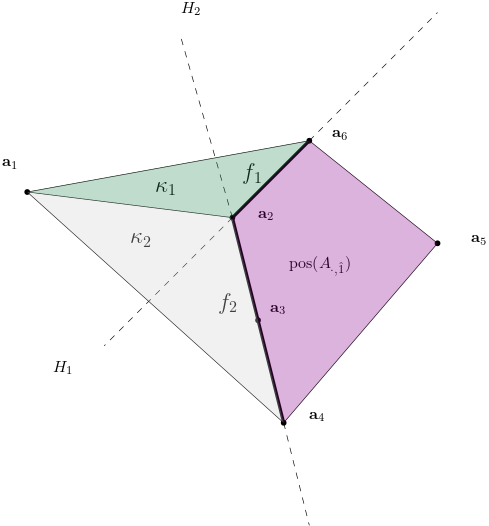}
    \caption[A sketch of some of the elements in the proof of Lemma \ref{lem:unique-facet}.]{A sketch of some of the elements in the proof of Lemma \ref{lem:unique-facet}.
    Here $A$ is a $3 \times 6$ matrix with columns $\BF{a}_1, \dots, \BF{a}_6$. 
    The polytope with vertices labelled $\BF{a}_1, \BF{a}_4, \bf{a}_5, \BF{a}_6$ represents a 2-dimensional cross-section of the 3-dimensional cone $\pos(A)$. Each of the points labelled by a column $\BF{a}_j$ of $A$ represent the ray generated by $\BF{a}_j.$ From the picture we see that $\BF{a}_1, \BF{a}_4, \bf{a}_5, \BF{a}_6$ are the external columns of $A$. The cone $\pos(A_{\cdot, \hat{1}})$ is shaded in dark magenta. The hyperplanes $H_1$ and $H_2$ separate $\BF{a}_1$ and $\pos(A_{\cdot, \hat{1}})$.
    Any chamber $\gamma$ containing $\BF{a}_1$ must be contained in one of the cones $\kappa_1$ (shaded dark green) or $\kappa_2$ (shaded light grey) and the unique facet of $\gamma$ not containing $\BF{a}_1$ is equal to $\gamma \cap f_1$ or $\gamma \cap f_2$.}
    \label{fig:unique-facet}
\end{figure}

We define the cone generated by the external ray generators of an external chamber $\gamma$ to be an \emph{external facet} of $A$. The following result illustrates how to identify external facets and compute the corresponding external chambers. We impose the somewhat artificial condition that $A$ should have at least two chambers to avoid the degenerate case referred to in Remark~\ref{rem:degenerate}.

\begin{prop}[Constructing external chambers] \label{prop:construct-external}
Assume that the chamber complex of $A$ contains at least two chambers, and let $f$ be a $(d-1)$-dimensional cone in the chamber complex. 
Then $f$ is an external facet of $A$ if and only if $f$ is a facet of $\pos(A)$ containing exactly $d-1$ columns of $A$. Moreover, if the columns of $A$ generating $f$ are $\BF{a}_1, \dots, \BF{a}_{d-1}$, then the unique external chamber containing $f$ is
\begin{equation*}
\gamma :=
\bigcap\limits_{k=0}^{n-d} \pos(\BF{a}_1, \dots, \BF{a}_{d-1}, \BF{a}_{d+k}).
\end{equation*}

\end{prop}

\begin{proof}
We begin with the forward direction. Suppose that $f$ is an external facet of $A$, and assume towards a contradiction that $f$ is not a facet of $\pos(A)$ containing exactly $d-1$ columns of $A$. Since $f$ is an external facet of $A$, it is a facet of an external chamber $\gamma'$ of $A$. Since $\gamma'$ is simplicial, it has $d$ facets, say $f_1, \dots, f_{d-1}, f_d = f$. Assume moreover that $f_i$ is the unique facet of $\gamma'$ not containing $\BF{a}_i$ for each $i=1, \dots, d-1$. Let $\BFG{\iota}_1, \dots, \BFG{\iota}_d$ be the inner facet normals (with respect to~$\gamma'$) corresponding to the facets $f_1, \dots, f_d$, and $H_1, \dots, H_d$ be the corresponding supporting hyperplanes. Since $f$ is not a facet of $A$ containing exactly $d-1$ columns of $A$, there are two options to consider
\begin{enumerate}
    \item $f$ is not a facet of $\pos(A)$,
    \item $f$ is a facet of $\pos(A)$, but contains more than $d-1$ columns of $A$.
\end{enumerate}

In the first case there is some column $\BF{a}$ of $A$ with $\BFG{\iota}_d \cdot \BF{a} \leq 0$ and $\BF{a} \notin \gamma$. By Lemma \ref{lem:unique-facet}, the hyperplane $H_j$ separates the column $\BF{a}_j$ from the cone $\pos(A_{\cdot, \hat{j}})$ for each $j=1, \dots, d-1$. Therefore, $\BFG{\iota}_j \cdot \BF{a} \leq 0$ for each $j=1, \dots, d-1$. Then $-\BF{a} \in \gamma',$ since $\BFG{\iota}_j \cdot (-\BF{a}) \geq 0$ for each $j=1, \dots, d$. Now, let $\pos(A_s)$ be a simplicial cone of $A$ for some $s \subseteq \{1, \dots, n\}$ with $|s| = \rank(A_s) = d$ so that $\gamma' \subseteq \pos(A_s).$ Since $-\BF{a} \in \gamma'$, $-\BF{a} \in \pos(A_s)$, and so $-\BF{a} = \sum_{i \in s} \lambda_i \BF{a}_i$ for some $\lambda_i \geq 0.$ But then, $\sum_{i \in s} \lambda_i \BF{a}_i + \BF{a} = \BF{0}$, and so $\ker(A) \cap \mathbb{R}^d_{\geq 0} \neq \{\BF{0}\}$. This is a contradiction. Therefore, $f$ is indeed a facet of $\pos(A)$, and $f$ contains exactly the $d-1$ columns of $A$, $\BF{a}_1, \dots, \BF{a}_{d-1}.$

In the second case, there is some column $\BF{a} \in f$ with $\BF{a} \notin \{\BF{a}_1, \dots \BF{a}_{d-1}\}$. However, since $\BF{a}$ is a column of $A$, it generates a 1-dimensional cone of the chamber complex of $A$, and so $f = \pos(\BF{a}_1, \dots \BF{a}_{d-1})$ cannot be a cone of the chamber complex. This contradicts the fact that $f$ is an external facet. 

We now prove the reverse direction. If $f$ is a facet of $\pos(A)$ containing exactly $d-1$ columns of $A$, then each of these columns is a ray generator of $\pos(A)$ and no pair is linearly dependent. Therefore, each of $\BF{a}_1, \dots, \BF{a}_{d-1}$ are external columns of~$A$.

Finally, we show that $\gamma$ is indeed a chamber of $A$. First note that none of the columns $\BF{a}_d, \dots, \BF{a}_n$ lie on $f$, and so $\pos(\BF{a}_1, \dots, \BF{a}_{d-1}, \BF{a}_{d+k})$ is a simplical cone of $A$ for each $k~=0,~\dots,n-d$. Therefore $\gamma$ is the intersection of simplicial cones, and since each of these simplicial cones lie on the same side of the facet $f$, the cone $\gamma$ must be $d$-dimensional. Consider the point 
\begin{equation} \label{eq:pos-hull-combination}
    \BF{b} := \BF{a}_1 + \dots + \BF{a}_{d-1}.
\end{equation}
 Since $\{\BF{a}_1, \dots, \BF{a}_{d-1}\}$ is a linearly independent set, and $f$ is a facet of $\pos(A)$, the formulation of \eqref{eq:pos-hull-combination} is the unique way to represent $\BF{b}$ as a $\mathbb{N}$-linear combination of the columns of $A$. Therefore, any simplicial cone $\pos(A_s)$ of $A$ containing $\BF{b}$ must contain each of the external columns $\BF{a}_1, \dots, \BF{a}_{d-1}$. Moreover, $\BF{a}_1, \dots, \BF{a}_{d-1}$ are each ray generators of $\pos(A_s)$ by Proposition~\ref{prop:external-simplicial}. Therefore, $\gamma$ is a chamber of $A$ since it is a $d$-dimensional cone obtained as the intersection of all simplicial cones containing $\BF{b}$. Furthermore, $\gamma$ is the unique external chamber containing $f$, since any other $d$-dimensional cone containing $\BF{b}$ and obtained by an intersection of simplicial cones of $A$ (necessarily containing all of $\BF{a}_1, \dots, \BF{a}_{d-1}$ as ray generators) must contain $\gamma$ as a subset.

 By Lemma \ref{lem:unique-facet}, there is a unique facet of $\gamma$ not containing $\BF{a}_i$ for each $1~\leq~i~\leq~d~-~1$, and so $\gamma$ is simplicial. Therefore, $\gamma = \pos(\BF{a}_1, \dots, \BF{a}_{d-1}, \BF{v})$ for some ray generator~$\BF{v}$. Since $A$ has at least two chambers, it must have some column $\BF{c} \notin \gamma$. As well, $\BF{c}~\notin~f$, so $\gamma'~:=~\pos(\BF{a}_1, \dots, \BF{a}_{d-1}, \BF{c})$ is a simplicial cone of~$A$, and $\gamma \subsetneq \gamma'$. Therefore, it follows that $\BF{v} \in \pos(\BF{a}_1, \dots, \BF{a}_{d-1}, \BF{c})$ and so $\BF{v}$ is an internal ray generator for $\gamma$. Thus, $\gamma$ is an external chamber of $A$ and $f$ is an external facet of $A$. 
\end{proof}

The previous result allows us to compute external chambers without having to compute the entire chamber complex of $A$, which can be computationally intensive. We note however that in some cases there are no external chambers. For example, for the matrix
\begin{equation*}
    K_4 = 
    \begin{bmatrix}
        1 & 1 & 1 & 1 & 0 & 0 & 0 & 0 & 0 & 0 \\
        -1 & 0 & 0 & 0 & 1 & 1 & 1 & 0 & 0 & 0 \\
        0 & -1 & 0 & 0 & -1 & 0 & 0 & 1 & 1 & 0 \\
        0 & 0 & -1 & 0 & 0 & -1 & 0 & -1 & 0 & 1
    \end{bmatrix}
\end{equation*}
the chamber complex has $48$ chambers, none of which are external. The matrix $K_4$ is part of the same family of matrices associated to Kostant's partition as $K_3$. 

Finally, we remark that external facets of $A$ are exactly the facets $f$ of $\pos(A)$ for which $p_A(\BF{b}) \leq 1$ for all $\BF{b} \in f$. 

\subsection{A vector partition function preserving transformation} \label{subsec:transformation}

We now prove some results that allow us to transform the matrix $A$ while preserving the vector partition function (up to an appropriate change of variables) and the structure of the chamber complex. We use these results in Section \ref{sec:dim-reduction} in order to transform $A$ into a form well-suited for analysis (described in Lemma~\ref{lem:dim-reduction}). Propositions \ref{prop:invertible-map} and \ref{prop:chamber} are straight-forward results, and have been assumed by other authors, so we give them without proof.

\begin{prop} \label{prop:invertible-map}
Let $M \in \mathbb{Q}^{d \times d}$ be an invertible matrix with integer entries. Then $p_A(\BF{b}) = p_{MA}(M\BF{b})$ for all $\BF{b} \in \mathbb{Z}^d$. 
\end{prop}

For a cone $\sigma \subseteq \mathbb{R}^d$ and invertible matrix $M \in \mathbb{Q}^{d \times d}$, define the cone $M\sigma := \{M\BF{b} : \BF{b} \in \sigma\}$. We note that $\{\BF{u}_1, \dots \BF{u}_k\}$ is a generating set of $\sigma$ if and only if $\{M\BF{u}_1, \dots M\BF{u}_k\}$ is a generating set of $M\sigma$. 

\begin{prop} \label{prop:chamber}
Let $M \in \mathbb{Q}^{d \times d}$ be an invertible matrix. The cone~$\gamma$ is a chamber of~$A$ if and only if  $M\gamma$ is a chamber of $MA$. Moreover, $\{\BF{u}_1, \dots, \BF{u}_k\}$ is a minimal generating set of $\gamma$ if and only if $\{M\BF{u}_1, \dots, M\BF{u}_k\}$ is a minimal generating set of $M\gamma$.
\end{prop}

\begin{rem}
The previous proposition does not always hold if we replace ``minimal generating set'' with ``minimal ray generators''. In the case that $\gamma$ is a simplicial cone with minimal ray generators $\BF{v}_1, \dots, \BF{v}_d$, then $M\BF{v}_1, \dots, M\BF{v}_d$ are minimal ray generators of $M\gamma$ if and only if $M$ is invertible over $\mathbb{Z}$ (equivalently $\det(M) = \pm 1$). 
\end{rem}

\begin{prop} \label{prop:external-column} 
Let $M \in \mathbb{Q}^{d \times d}$ be an invertible matrix. Let $\BF{a}_j$ be a column of $A$ for some $j~\in~\{1, \dots, n\}$. Then
$\BF{a}_j$ is an external column of $A$ if and only if $M\BF{a}_j$ is an external column of $MA$. 
\end{prop}

\begin{proof}
Assume without loss of generality that $j=1$.
Points $\BF{b}$ in $\pos(A_{\cdot, \hat{1}})$ map to points in $\pos(MA_{\cdot, \hat{1}})$ under the invertible mapping $\BF{b} \mapsto M\BF{b}$:
\begin{align*}
\BF{b} \in \pos(A_{\cdot, \hat{1}}) &\iff \BF{b} = \sum\limits_{i=2}^{n}\lambda_i\BF{a}_i & (\text{for } \lambda_2, \dots, \lambda_n \geq 0) \\
& \iff M\BF{b} = \sum\limits_{i=2}^{n}\lambda_iM\BF{a}_i \\
& \iff M\BF{b} \in \pos(MA_{\cdot, \hat{1}})
\end{align*}
Therefore $\BF{a}_1 \notin \pos(A_{\cdot, \hat{1}})$ if and only if $M\BF{a}_1 \notin \pos(MA_{\cdot, \hat{1}})$, and so $\BF{a}_1$ is an external column if and only if $M\BF{a}_1$ is an external column.
\end{proof}

\begin{prop} \label{prop:external-chamber}
Let $M \in \mathbb{Q}^{d \times d}$ be an invertible matrix. Then
$\gamma$ is an external chamber of $A$ if and only if $M\gamma$ is an external chamber of $MA$. 
\end{prop}

\begin{proof}
This follows immediately from Propositions \ref{prop:chamber} and \ref{prop:external-column}.
\end{proof}

In the following section, we show that one can always construct an invertible matrix $M \in \mathbb{Q}^{d \times d}$ with integer entries that maps each of the external columns of $A$ in $\gamma$ to positive multiples of the standard basis vectors. From this form we are able to treat certain variables as slack variables in order to reduce dimension.

\section{Reduction of dimension} \label{sec:dim-reduction}

In this section, we consider chambers of $A$ that contain external columns. Up to a lattice condition, we show that the quasi-polynomial $p_A^{\gamma}$ for such a chamber $\gamma$ can be obtained via a vector partition function of lower dimension. In particular, if $\gamma$ has $k$ external columns, then (up to a change of variables) $p_A^{\gamma} = p_B^{\gamma'}$ for a matrix $B$ of $k$ fewer rows and columns than $A$, and chamber $\gamma'$ of $B.$ When this result is applied to external chambers, we find that $B$ has a single row, so that $p^{\gamma}_A$ is obtained from a coin exchange problem. As a consequence, such a $p_A^{\gamma}$ is a univariate quasi-polynomial. Indeed, $p_A^{\gamma}$ is precisely the Ehrhart quasi-polynomial associated to the internal ray of $\gamma$.

Let $\lattice(A)$ denote the lattice generated by the columns of $A$, and $\semigroup{A}$ denote the affine semigroup generated by the columns of $A$ -- that is, 
\begin{align*}
    \lattice(A) := \left\{\sum_{i=1}^{n} \lambda_i\BF{a}_i : \lambda_i \in \mathbb{Z}\right\}, \\
    \semigroup{A} := \left\{\sum_{i=1}^{n}\lambda_i\BF{a}_i : \lambda_i \in \mathbb{N}\right\}.
\end{align*}

An affine semigroup $\mathcal{S}$ is \emph{saturated} in a lattice $\mathcal{L}$ if for each positive integer $k$ and $\BF{v} \in \mathcal{L}$, $k\BF{v} \in \mathcal{S}$ only if $\BF{v} \in \mathcal{S}$. In our setting, we are generally interested in showing that affine semigroup generated by a subset of columns of $A$ is saturated in the lattice generated by the columns of $A$ -- that is, that $\semigroup{A_s}$ is saturated in the lattice $\mathcal{L}(A)$ for some $s \subseteq \{1, 2, \dots, n\}$. This is equivalent to the condition that
\begin{equation}
    \semigroup{A_s} = \lattice(A) \cap \pos(A_s)
\end{equation}
for which it is sufficient to show prove the reverse inclusion $\semigroup{A_s} \supseteq \lattice(A) \cap \pos(A_s)$\footnote{See \cite[Proposition 1.1]{Od88} for details.}.

The following technical lemma shows that under certain conditions a quasi-polynomial $p_A^{\gamma}$ can be obtained from the vector partition function of $B$ for a submatrix $B$ of $A$. The proof involves several intermediate results and is worked out in detail in Section \ref{sec:lem-proof}. 

\begin{lem}  \label{lem:dim-reduction}
Let $\gamma$ be a chamber of $A$. Assume without loss of generality that the external columns of $A$ in $\gamma$ are $\BF{a}_1, \dots, \BF{a}_{\ell}$ for some $\ell \in \{0,\dots, d-1\}$. Also assume that $\BF{a}_i = k_i\BF{e}_i$ for each $i \in \{1, \dots, \ell\}$ and some positive integers $k_1, \dots, k_{\ell}.$ Finally assume that $\semigroup{\BF{a}_1, \dots, \BF{a}_{\ell}}$ is saturated in $\mathcal{L}(A)$.
Let $B$ be the matrix obtained by removing the first $\ell$ rows and columns of $A$. 
Then there exists a chamber $\gamma'$ of $B$ such that
\begin{equation*} 
p^{\gamma}_A(\BF{b}) = p_{B}^{\gamma'}(b_{\ell+1}, \dots, b_d) 
\end{equation*}
for all $\BF{b} = (b_1, \dots, b_n) \in \semigroup{A} \cap \gamma$.
\end{lem}

In the previous lemma, the condition that the external columns $\gamma$ are given by positive integer multiples of standard basis vectors may appear contrived.
However, we now show that if $A$ has a chamber $\gamma$ whose external columns generate an affine semigroup saturated in $\lattice(A)$, one can always apply an appropriate change of variables so that the pair $A$ and $\gamma$ are in the form of Lemma~\ref{lem:dim-reduction}.

For a cone $\sigma$ fix an ordering $\BF{v}_1, \dots, \BF{v}_m$ of the minimal ray generators of $\sigma$. We define the \emph{ray matrix} $M_{\sigma}$ of $\sigma$ to be the matrix whose rows are the minimal ray generators of $\sigma$. If $\sigma$ is a simplicial cone with minimal ray generators $\BF{v}_1, \BF{v}_2, \dots, \BF{v}_d$, let $\BF{w}_1, \dots, \BF{w}_d$ be minimal ray generators of $\sigma^{\vee}$ so that $\BF{w}_i$ is the minimal inner facet normal of the sole facet of $\sigma$ not containing $\BF{v}_i$. We abuse notation by setting $M_{\sigma^{\vee}}$ to be the matrix whose rows are $\BF{w}_1, \dots, \BF{w}_d$ (so that the order of rows of $M_{\sigma^{\vee}}$ is set by the ordering of minimal ray generators of $\sigma$). We call $M_{\sigma^{\vee}}$ the \emph{dual ray matrix} of $\sigma$. We remark that for each $i=1, \dots, d$
\begin{align*}
    M_{\sigma^{\vee}}\BF{v}_i &= 
    \begin{bmatrix}
         \BF{w}_1 \cdot \BF{v}_i \\
         \BF{w}_2 \cdot \BF{v}_i \\
        \vdots \\
        \BF{w}_d \cdot \BF{v}_i
    \end{bmatrix} \\
    &= k_i\BF{e}_i
\end{align*}
for some positive integer $k_i$.

\begin{theo} \label{theo:external-col-variables}
    Let $A$ be a $d \times n$ matrix of rank $d$ with integer entries, and let $\gamma$ be a chamber of $A$. Without loss of generality assume that $\BF{a}_1, \dots, \BF{a}_{\ell}$ are the external columns of $\gamma$. Assume additionally that $\semigroup{\BF{a}_1, \dots, \BF{a}_{\ell}}$ is saturated in $\lattice(A)$. Let $\sigma$ be a simplicial cone of $A$ containing $\gamma$, and let $B$ be the matrix obtained by removing the first $\ell$ rows and columns from $M_{\sigma^{\vee}}A$.
    Then 
    \begin{equation*}
        p^{\gamma}_A(\BF{b}) = p^{\gamma'}_{B}\left((M_{\sigma^{\vee}}\BF{b})_{\ell + 1}, \dots, (M_{\sigma^{\vee}}\BF{b})_{d}\right)
    \end{equation*}
    for all $\BF{b} \in \gamma \cap \semigroup{A}$. 
\end{theo}

\begin{proof}
    Since $\gamma$ is a chamber of $A$, it is contained in some simplicial cone $\sigma$ of $A$, and so the dual ray matrix $M := M_{\sigma^{\vee}}$ is well defined. Additionally, each of $\BF{a}_1, \dots, \BF{a}_{\ell}$ are ray generators of $\sigma$ since they are external columns of $A$, and are each in $\sigma$. Let $\BF{v}_1, \dots, \BF{v}_{\ell}$ be minimal ray generators of $\sigma$ in the same direction as $\BF{a}_1, \dots, \BF{a}_{\ell}$ respectively.
    Without loss of generality assume that row $i$ of $M_{\sigma^{\vee}}$ corresponds to minimal ray generator $\BF{v}_i$ for each $i =1, \dots, \ell$ (ordering the remainder of the rows arbitrarily).
    Consider the matrix $MA$ with columns $\BF{m}_1 := M\BF{a}_1, \dots, \BF{m}_n := M\BF{a}_n$. Our goal is to apply Lemma~\ref{lem:dim-reduction} with the matrix $MA$, chamber $M\gamma$, and columns $\BF{m}_1, \dots, \BF{m}_{\ell}$ so we show that each of its conditions are met.

    The matrix $M_{\sigma^{\vee}}$ is invertible over $\mathbb{Q}$ and has integer entries, and thus satisfies the conditions of Propositions \ref{prop:invertible-map}--\ref{prop:external-chamber}. Therefore, the vector partition functions of $A$ and $M_{\sigma^{\vee}}A$ are the same up to a change of variables. More precisely, 
    \begin{equation} \label{eq:vpf-invertible}
        p_A^{\gamma}(\BF{b}) = p_{MA}^{M\gamma}(M\BF{b})
    \end{equation}
    for all $\BF{b} \in \gamma \cap \semigroup{A}.$
    
    By Proposition \ref{prop:chamber}, $M\gamma$ is a chamber of $MA$, and  $\BF{m}_1, \dots, \BF{m}_{\ell}$ are ray generators of~$M\gamma$. By Proposition~\ref{prop:external-column} they are each external columns of $MA$. Additionally, $\semigroup{\BF{m}_1, \dots, \BF{m}_{\ell}}$ is saturated in $\lattice(MA)$. Since $\BF{a}_1, \dots, \BF{a}_{\ell}$ are positive integer multiples of $\BF{v}_1, \dots, \BF{v}_{\ell}$ respectively, by our final remark before this theorem, it follows that $\BF{m}_1 = M\BF{a}_1 = k_1\BF{e}_1, \dots, \BF{m}_{\ell} = M\BF{a}_{\ell} = k_{\ell}\BF{e}_{\ell}$ for some positive integers $k_1, \dots, k_{\ell}$. 
    
    Therefore $MA$, $M\gamma$, and $\BF{m}_1, \dots, \BF{m}_{\ell}$ do indeed meet the conditions of Lemma \ref{lem:dim-reduction}, and so 
    \begin{equation} \label{eq:vpf-decompose}
        p^{M\gamma}_{MA}(M\BF{b}) = p^{\gamma'}_{B}((M\BF{b})_{\ell+1}, \dots, (M\BF{b})_{d})
    \end{equation}
    for each $\BF{b}~\in~\gamma~\cap~\semigroup{A}$. Putting together \eqref{eq:vpf-invertible} and \eqref{eq:vpf-decompose} yields the result. 
\end{proof}

In essence, (up to saturation) one can reduce the dimension of the vector partition function for a particular chamber by the number of external columns present in that chamber. 

If the chamber $\gamma$ is simplicial, one can obtain a slightly nicer result. By replacing $M_{\sigma}^{\vee}$ with $M_{\gamma}^{\vee}$ (and thus also appropriately updating $B$ and $\gamma'$) in Theorem \ref{theo:external-col-variables}, we obtain that $\gamma'$ is the positive orthant $\mathbb{R}^{d - \ell}$ (see Section \ref{sec:lem-proof}, Lemma \ref{lem:chamber-dimension-lower} for details). That is, $\gamma'$ is the cone defined by the inequalities $M_{\gamma^{\vee}}\BF{b}_{\ell+1}, \dots, M_{\gamma^{\vee}}\BF{b}_{d} \geq 0.$

We now give an example to illustrate our main result.

\begin{eg} \label{eg:main-theorem}
    Consider the matrix 
    \begin{equation*}
        G_6 = \left(\begin{array}{rrrrrrrrrrrrrrr}
1 & 1 & 1 & 1 & 1 & 0 & 0 & 0 & 0 & 0 & 0 & 0 & 0 & 0 & 0 \\
1 & 0 & 0 & 0 & 0 & 1 & 1 & 1 & 1 & 0 & 0 & 0 & 0 & 0 & 0 \\
0 & 1 & 0 & 0 & 0 & 1 & 0 & 0 & 0 & 1 & 1 & 1 & 0 & 0 & 0 \\
0 & 0 & 1 & 0 & 0 & 0 & 1 & 0 & 0 & 1 & 0 & 0 & 1 & 1 & 0 \\
0 & 0 & 0 & 1 & 0 & 0 & 0 & 1 & 0 & 0 & 1 & 0 & 1 & 0 & 1 \\
0 & 0 & 0 & 0 & 1 & 0 & 0 & 0 & 1 & 0 & 0 & 1 & 0 & 1 & 1
\end{array}\right)
    \end{equation*}
    and the chamber 
    \begin{equation*}
        \gamma = \pos\left(
        \BF{v}_1 := 
        \begin{bmatrix}
            1 \\ 0 \\ 0 \\ 0 \\ 1 \\ 0
        \end{bmatrix},\
        \BF{v}_2 := 
        \begin{bmatrix}
            1 \\ 0 \\ 0 \\ 1 \\ 0 \\ 0
        \end{bmatrix},\
        \BF{v}_3 :=
        \begin{bmatrix}
            1 \\ 0 \\ 1 \\ 0 \\ 0 \\ 0
        \end{bmatrix},\
        \BF{v}_4 :=
        \begin{bmatrix}
            1 \\ 1 \\ 0 \\ 0 \\ 0 \\ 0 
        \end{bmatrix},\
        \BF{v}_5 :=
        \begin{bmatrix}
            2 \\ 1 \\ 1 \\ 1 \\ 1 \\ 0
        \end{bmatrix},\
        \BF{v}_6 :=
        \begin{bmatrix}
            3 \\ 1 \\ 1 \\ 1 \\ 1 \\ 1
        \end{bmatrix}
        \right).
    \end{equation*}
    The first four ray generators of $\gamma$ are external, and generate an affine semigroup that is saturated in $\lattice(G_6)$.
    Therefore, by application of Theorem \ref{theo:external-col-variables}, we are able to compute $p_{G_6}^{\gamma}(\BF{d})$ (our choice of variable relates to the enumeration of multigraphs) by computing the vector partition function for a matrix with four fewer rows and columns. We now compute such a matrix. Since $\gamma$ is simplicial, we use the transformation \begin{equation*}
        M_{\gamma^{\vee}} := \begin{bmatrix} 
        1 & -1 & -1 & -1 & 1 & -1 \\
        1 & -1 & -1 & 1 & -1 & -1 \\
        1 & -1 & 1 & -1 & -1 & -1 \\
        1 & 1 & -1 & -1 & -1 & -1 \\
        -1 & 1 & 1 & 1 & 1 & -1 \\
        0 & 0 & 0 & 0 & 0 & 1
        \end{bmatrix}
    \end{equation*}
    (where the $i^{th}$ row of $M_{\gamma^{\vee}}$ is the minimal inner normal of the sole facet of $\gamma$ not containing $\BF{v}_i$). 
    We then compute that
    \begin{equation*}
        M_{\gamma^{\vee}}G_6 = \begin{bmatrix}
            0 & 0 & 0 & 2 & 0 & -2 & -2 & 0 & -2 & -2 & 0 & -2 & 0 & -2 & 0 \\
0 & 0 & 2 & 0 & 0 & -2 & 0 & -2 & -2 & 0 & -2 & -2 & 0 & 0 & -2 \\
0 & 2 & 0 & 0 & 0 & 0 & -2 & -2 & -2 & 0 & 0 & 0 & -2 & -2 & -2 \\
2 & 0 & 0 & 0 & 0 & 0 & 0 & 0 & 0 & -2 & -2 & -2 & -2 & -2 & -2 \\
0 & 0 & 0 & 0 & -2 & 2 & 2 & 2 & 0 & 2 & 2 & 0 & 2 & 0 & 0 \\
0 & 0 & 0 & 0 & 1 & 0 & 0 & 0 & 1 & 0 & 0 & 1 & 0 & 1 & 1
        \end{bmatrix}
    \end{equation*}
    and so, by removing the first 4 rows and columns of $M_{\gamma^{\vee}}G_6$, we obtain the matrix 
    \begin{equation*}
        B := \begin{bmatrix}
-2 & 2 & 2 & 2 & 0 & 2 & 2 & 0 & 2 & 0 & 0 \\
1 & 0 & 0 & 0 & 1 & 0 & 0 & 1 & 0 & 1 & 1
    \end{bmatrix}.
     \end{equation*}
     By our previous observation, we have that
     \begin{equation*}
         p_A^{\gamma}(\BF{d}) = p_B^{\gamma'}((M_{\gamma^{\vee}}G_6\BF{d})_5, ((M_{\gamma^{\vee}}G_6\BF{d})_6) = p_B^{\gamma'}(-d_1 + d_2 + d_3 + d_4 + d_5 - d_6, d_6)
     \end{equation*}
     where $\gamma'$ is the positive quadrant in $\mathbb{R}^2$ generated by $\BF{e}_1, \BF{e}_2$. 
     Finally, we compute that 
     \begin{equation*}
         p_B^{\gamma'}(\BF{b}) = \begin{cases}
             h(\BF{b}) &\text{ if } b_1 \equiv 0 \mod 2 \\
             0 &\text{ if } b_1 \equiv 1 \mod 2
         \end{cases}
     \end{equation*}
     where 
     \begin{dmath*}
     h(\BF{b}) = 
             \left(\frac{1}{5806080}\right) (b_{2} + 1)  (b_{2} + 2) (b_{2} + 3) (b_{2} + 4) (63 b_{1}^{5} + 126 b_{1}^{4} b_{2} + 168 b_{1}^{3} b_{2}^{2} + 144 b_{1}^{2} b_{2}^{3} + 72 b_{1} b_{2}^{4} + 16 b_{2}^{5} + 1890 b_{1}^{4} + 3360 b_{1}^{3} b_{2} + 3600 b_{1}^{2} b_{2}^{2} + 2160 b_{1} b_{2}^{3} + 560 b_{2}^{4} + 21420 b_{1}^{3} + 32040 b_{1}^{2} b_{2} + 24600 b_{1} b_{2}^{2} + 7760 b_{2}^{3} + 113400 b_{1}^{2} + 128400 b_{1} b_{2} + 53200 b_{2}^{2} + 276192 b_{1} + 180384 b_{2} + 241920).
     \end{dmath*}
     Therefore, we find that 
     \begin{equation*}
         p_A^{\gamma}(\BF{d}) = \begin{cases}
             h\left(- d_1 + d_2 + d_3 + d_4 + d_5 - d_6, d_6 \right) &\text{ if } d_1 + d_2 + d_3 + d_4 + d_5 + d_6 \equiv 0 \mod 2 \\
             0 &\text{ if } d_1 + d_2 + d_3 + d_4 + d_5 + d_6 \equiv 1 \mod 2
         \end{cases}
     \end{equation*}
     where we have exploited the property that $d_1 + d_2 + d_3 + d_4 + d_5 + d_6 \equiv -d_1 + d_2 + d_3 + d_4 + d_5 - d_6 \mod 2$.
     We study the matrix $G_6$, which is in a class of matrices related to the enumeration of multigraphs, in more detail in Section \ref{sec:multigraph}. We complete this example by remarking that $p_A^{\gamma}(\BF{d})$ enumerates the number of loopless multigraphs on vertex set $v_1, \dots, v_6$ such that the degree of vertex $v_i$ is $d_i$ and the degrees satisfy
     \begin{align*}
        d_1 + d_j \geq \sum_{\substack{i = 2 \\ i \neq j}}^{6} d_i, \\
        d_1 + d_6 \leq \sum_{i = 2}^{5} d_i.
     \end{align*}
\end{eg}

\section{External chamber case} \label{sec:external-chamber-case}

If $\gamma$ is an external chamber whose external columns generate an affine semigroup saturated in $\lattice(A)$, then $B$ is a $1 \times (n-d)$ matrix, and so the quasi-polynomial $p_A^{\gamma}$ arises from a coin exchange problem.
By exploiting this fact, we prove that $p_A^{\gamma}$ can also be obtained from the Ehrhart quasipolynomial associated to the single internal ray of $\gamma$ after an appropriate change of variables. 

\subsection{Determinantal formula} \label{subsec:coin-exchange}

\begin{theo} \label{theo:ehrhart-det}
Let $A$ be a $d \times n$ matrix of rank $d$ with integer entries. Let $\gamma$ be an external chamber of $A$, and without loss of generality assume that the external columns of~$\gamma$ are $\BF{a}_1, \dots, \BF{a}_{d-1}$. Assume additionally that $\semigroup{\BF{a}_1, \dots, \BF{a}_{d-1}}$ is saturated in $\lattice(A)$. Denote by 
$\BF{v}_1, \BF{v}_2, \dots, \BF{v}_{d-1} \in \mathbb{Z}^d$ the external ray generators corresponding to $\BF{a}_1, \dots, \BF{a}_{d-1}$ respectively, and let $\BF{v}_d \in \mathbb{Z}^d$ be an internal ray generator.
If $h(t) := p_A(t\BF{v}_d)$ is the Ehrhart quasi-polynomial associated to the polytope $A\BF{x} = \BF{v}_d$, $\BF{x} \geq \BF{0}$, then the quasi-polynomial $p_A^{\gamma}(\BF{b})$ associated to $\gamma$ is equal to 
\begin{equation*}
    p_A^{\gamma}(\BF{b}) = h\left(\frac{\det(\BF{v}_1, \dots, \BF{v}_{d-1}, \BF{b})}{\det(\BF{v}_1, \BF{v}_2, \dots, \BF{v}_{d})}\right)
\end{equation*}
for all $\BF{b} \in \gamma \cap \semigroup{A}$. 
\end{theo}

\begin{proof}
Since $\gamma$ is an external chamber, by Proposition \ref{prop:chamber-simplicial}, it is simplicial, and so the dual ray matrix of $\gamma$ exists. Let $M := M_{\gamma^{\vee}}$ following the same ordering as the ray generators $\BF{v}_1, \dots, \BF{v}_d$. Let $\BF{b} \in \gamma \cap \semigroup{A}$. Then 
\begin{equation} \label{eq:b-lin-coeff}
    \BF{b} = \lambda_1\BF{v}_1 + \dots + \lambda_d\BF{v}_d 
\end{equation}
for some $\lambda_1, \dots, \lambda_d \geq 0$, and so
\begin{equation*}
    M\BF{b} = \lambda_1k_1\BF{e}_1 + \dots + \lambda_dk_d\BF{e}_d.
\end{equation*}
By Theorem \ref{theo:external-col-variables}, we have
\begin{equation*}
    p_A^{\gamma}(\BF{b}) = p_B(\lambda_dk_d)
\end{equation*}
where $B$ is the $1 \times (n-d+1)$ matrix obtained by removing the first $d-1$ rows and columns of the matrix $MA$. On the other hand, by setting $\lambda_1 = \ldots = \lambda_{d-1} = 0$ in Eq.~\eqref{eq:b-lin-coeff}, we find that $p_A^{\gamma}(\lambda_d\BF{v}_d) = p_B(\lambda_dk_d)$ as well. Therefore, $p^{\gamma}_A(\BF{b}) = p^{\gamma}_A(\lambda_d\BF{v}_d) = h(\lambda_d)$. Finally, by Cramer's rule
\begin{equation*}
    \lambda_d = \frac{\det(\BF{v}_1, \dots, \BF{v}_{d-1}, \BF{b})}{\det(\BF{v}_1, \BF{v}_2, \dots, \BF{v}_d)}
\end{equation*}
and so as quasi-polynomials,
\begin{equation*}
    p_A^{\gamma}(\BF{b}) = h\left(\frac{\det(\BF{v}_1, \dots, \BF{v}_{d-1}, \BF{b})}{\det(\BF{v}_1, \BF{v}_2, \dots, \BF{v}_{d}))}\right)
\end{equation*}
as required.

\end{proof}

\begin{eg} \label{eg:periodicity}
Recall the matrix 
\[
A^{2,2} = \begin{bmatrix}
    1 & 0 & 1 & 1 \\
    0 & 1 & 1 & 2
\end{bmatrix}
\]
whose columns we denote by $\BF{a}_1, \BF{a}_2, \BF{a}_3, \BF{a}_4$.
The chamber 
\[
\gamma_1 = \pos(\BF{a}_1, \BF{a}_3)
\]
is an external chamber with external column $\BF{a}_1$. From Figure \ref{fig:eg1}, we can see that the column $\BF{a}_3$ is an internal ray generator for $\gamma_1$. In particular we can take $\BF{v}_1 := \BF{a}_1$ to be the external ray generator and $\BF{v}_2 := \BF{a}_3$ to be the internal ray generator. Let $h(t) = p_{A^{2,2}}\left(t\BF{a}_3\right) $ be the Ehrhart quasi-polynomial associated to the internal ray of $\gamma_1$.
We can compute using \emph{Latte}:
\begin{equation*}
    h(t) = \begin{cases}
        \frac{(t+2)^2}{4} \quad &\text{ if } t \equiv 0 \mod 2 \\
        \frac{(t+1)(t+3)}{4} \quad &\text{ if } t \equiv 1 \mod 2.
    \end{cases}
\end{equation*}
By Theorem \ref{theo:ehrhart-det} we deduce that
\begin{align*}
p^{\gamma_1}_{A^{2,2}}(\BF{b}) &= 
h\left(\frac{\det\left(
\BF{a}_1, \BF{b}
\right)}
{\det\left(
\BF{a}_1, \BF{a}_2
\right)}\right) \\
&= h\left(\frac{\det\left(
\begin{bmatrix}
    1 & b_1 \\
    0 & b_2
\end{bmatrix}
\right)}
{\det\left(
\begin{bmatrix}
    1 & 1 \\
    0 & 1
\end{bmatrix}
\right)}\right) \\
&= \begin{cases}
        \frac{(b_2+2)^2}{4} \quad &\text{ if } b_2 \equiv 0 \mod 2 \\
        \frac{(b_2+1)(b_2+3)}{4} \quad &\text{ if } b_2 \equiv 1 \mod 2.
    \end{cases}
\end{align*} 
This agrees with previous computations \cite{MiRoSu21}, and the output of \emph{Barvinok}.
\end{eg}

In the previous example, we could have also applied Theorem \ref{theo:external-col-variables} to prove that for all $\BF{b} \in \gamma_3$,
\begin{equation*}
    p^{\gamma_1}_{A^{2,2}}(\BF{b}) = p_B(b_2)
\end{equation*}
where $B = \begin{bmatrix}
    1 & 1 & 2
\end{bmatrix}$, and then solved the corresponding coin exchange problem. Example~\ref{eg:multigraph} in Section \ref{sec:multigraph} provides a slightly more involved application of the determinant formula.

\subsection{Polynomiality}

In the previous section we showed that if $\gamma$ is an external chamber of $A$ whose external columns generate a saturated affine semi-group in $\lattice(A)$, then the quasi-polynomial $p_A^{\gamma}$ is equal to $p_B$ for a $1 \times k$ matrix $B$ with integer entries. Next, we exploit this this fact in order to characterize exactly when $p_A^{\gamma}$ is a polynomial. Moreover, we show that this polynomial is given by a negative binomial coefficient and is easy to compute, without explicitly computing the chamber $\gamma$. For a class of matrices (called \emph{unimodular} matrices), this result immediately allows us to prove that the polynomial $p_A^{\gamma}$ for an external chamber $\gamma$ is given by a negative binomial coefficient that is readily computable.

\begin{lem} \label{lem:coin-exchange}
Let $B = [b_{1,1}, \dots, b_{1,k}]$ be a $1 \times k$ integer matrix for some positive integer $k$, and assume that $\ker(B) \cap \mathbb{R}^k_{\geq 0} = \{\BF{0}\}$. Then $p_B$ is a polynomial of degree $k-1$ on $\semigroup{B}$ if and only if each of the $k$ entries of $B$ are equal to some non-zero integer $\beta$. In this case, 
\begin{equation*}
    p_B(b) =  \binom{\frac{b}{\beta} + k - 1}{k - 1}
\end{equation*}
for all $b \in \semigroup{B}$.
\end{lem}

\begin{proof}
We begin by proving the reverse implication. Assume $B$ is a $1 \times k$ integer matrix with each of the $k$ entries equal to some non-zero integer $\beta$. Then for each $b \in \semigroup{B}$, $p_B(b)$ is the number of ways of partitioning $b/{\beta}$ into $k$ equal non-negative integral parts. Therefore, 
\begin{equation} \label{eq:bin-coeff}
p_B(b)  = \binom{\frac{b}{\beta} + k - 1}{k - 1}
\end{equation}
is a polynomial in $b$ for $b \in \semigroup{B}$. 

We now prove the forward implication. 
Suppose $p_B$ is a polynomial of degree $k-1$. We may assume that $k \geq 2$, since if $k=1$, $B$ has a single entry. We note further that the entries of $B$ must be either all positive or all negative or else $\ker(B) \cap \mathbb{R}^k_{\geq 0} \neq \{\BF{0}\}$. We assume that all entries are positive, noting that the negative case follows a similar argument. For all $1 \leq j \leq k$, the vector partition function $p_{B_{\cdot, \hat{j}}}$ is a polynomial of degree $k-2$ since it is the difference of two polynomials:
\begin{equation*}
p_{B_{\cdot, \hat{j}}}(b) = p_B(b) - p_B(b - b_{1,j}) 
\end{equation*}
for all $b \in \lattice(B) \cap \mathbb{N}$. In particular, by repeated application of this fact, it follows that for each $1 \times 2$ submatrix of $B$, the vector partition function is a polynomial of degree $1$. 

Assume towards a contradiction that $B$ has two distinct entries, say, without loss of generality, $b_{1,1}$ and $b_{1,2}$. Let $B' = [b_{1,1}, b_{1,2}]$ be the $1 \times 2$ submatrix consisting of the two distinct entries, so that $p_{B'}$ is a polynomial of degree $1$. 
Also, $p_{B'}(0) = p_B(\min(b_{1,1}, b_{1,2})) =1$, so $p_{B'}(b) = 1$ for all $b \in \mathbb{N}$ since $p_{B'}$ is linear. However, $p_{B'}(b_{1,1}b_{1,2}) \geq 2$ since both $\BF{x} = (b_{1,2}, 0)$ and $\BF{x} = (0, b_{1,1})$ are solutions to $B'\BF{x} = b_{1,1}b_{1,2}$ with $\BF{x} \in \mathbb{N}^2$. This contradicts that $p_{B'}$ is a polynomial of degree $1$, and thus that $p_B$ is a polynomial of degree $k-1$. Therefore, the entries of $B$ must be the same as required.

\end{proof}

 For a facet $f$ of a cone $\sigma \subset \mathbb{R}^m$, we call an inner/outer facet normal $\BFG{\iota} \in \mathbb{Z}^m$ of $f$ a \emph{minimal inner/outer facet normal} if $\BFG{\iota}$ is a minimal generator of the ray $\{t\BFG{\iota} : t \geq 0\}$. By Proposition \ref{prop:facet-normal-dual-ray}, $\BFG{\iota}$ is a ray generator of $\sigma^{\vee}$. Therefore, if $f$ is a facet of a simplicial chamber~$\gamma$ of $A$, then $\BFG{\iota}$ is a row of the dual ray matrix $M_{\gamma^{\vee}}$. This observation allows us to characterize exactly when $p_A^{\gamma}$ is a polynomial on $\semigroup{A}$ if $\gamma$ is an external chamber whose external columns generate an affine semigroup that is saturated in $\lattice(A)$.

 \begin{theo} \label{theo:external-chamber-binom}
 Let $\gamma$ be an external chamber of $A$, with external facet $f$, and let $\BFG{\iota}$ be the minimal inner facet normal of $f$. Assume without loss of generality that $\BF{a}_1, \dots, \BF{a}_{d-1}$ are the external columns of $\gamma$. Let $\BF{a}_{d+\ell}$ be a column of $A$ for some $\ell \in \{0, \dots, n-d\}$. Finally assume that $\semigroup{\BF{a}_1, \dots, \BF{a}_{d-1}}$ is saturated in~$\lattice(A)$.
     Then $p_A^{\gamma}$ is a polynomial on $\gamma \cap \semigroup{A}$ if and only if
     \begin{equation*}
         \BFG{\iota} \cdot \BF{a}_j = \begin{cases}
             0 \text{ if } \BF{a}_j \in f\\
             \beta \text{ if } \BF{a}_j \not \in f
         \end{cases}
     \end{equation*}
     for each $j = 1, \dots, n$, for some positive integer $\beta$. Moreover, if $p_A^{\gamma}$ is a polynomial on $\semigroup{A}$, then 
     \begin{align}
         p_A^{\gamma}(\BF{b}) &= \binom{\frac{\BFG{\iota} \cdot \BF{b}}{\beta} + n -d}{n-d} \label{eq:binom-dot-product} \\
         &= \binom{\frac{\det(\BF{a}_1, \dots, \BF{a}_{d-1}, \BF{b})}{\det(\BF{a}_1, \dots, \BF{a}_{d-1}, \BF{a}_{d+\ell})} + n -d}{n-d} \label{eq:binom-determinant}
     \end{align}
     for each $\BF{b} \in \gamma \cap \semigroup{A}$.
 \end{theo}

 \begin{proof}
     We have $\BFG{\iota} \cdot \BF{a}_j = 0$ for each $j= 1, \dots, d-1$.
     Let $M := M_{\gamma^{\vee}}$ be the dual ray matrix of $\gamma$ so that the first $d-1$ rows appear in the same order as the corresponding $d-1$ external columns of $\gamma$. 
     By Theorem \ref{theo:external-col-variables}, for each $\BF{b} \in \gamma$,
     \begin{equation*}
     p_A^{\gamma}(\BF{b}) = p_B((M\BF{b})_d)
     \end{equation*}
     where $B$ is the $1 \times (n-d +1)$ matrix obtained by removing the first $d-1$ rows and columns from~$MA$. Additionally, the last row of $M$ is simply $\BFG{\iota}$ since $\BFG{\iota}$ is the only minimal ray generator of $\gamma^{\vee}$ not corresponding to a column in $f$. Therefore, the last row of $MA$ is $\BFG{\iota}^TA$, and so we have
     \begin{align*}
         B_{1,j} &= (MA)_{d, d-1 + j} \\
         &= \BFG{\iota} \cdot \BF{a}_{d-1+j}
     \end{align*}
     for each $j = 1, \dots, n-d+1$.
     By Lemma \ref{lem:coin-exchange}, $p_B$ is polynomial on $\semigroup{B}$ if and only if each of these entries is equal to some positive integer $\beta$. Since $p_A^{\gamma}(\BF{b}) = p_B((M\BF{b})_d)$ for all $\BF{b} \in \gamma \cap \semigroup{A}$, $p_A^{\gamma}$ is polynomial if and only if $p_B$ is polynomial. 
     
     We now prove that Eq.~\eqref{eq:binom-dot-product} and Eq.~\eqref{eq:binom-determinant} hold if $p_A^{\gamma}$ is polynomial. In this case, for all $\BF{b} \in \gamma \cap \semigroup{A}$, 
     \begin{align*}
         p_A^{\gamma}(\BF{b}) &= p_B((M\BF{b})_d) \\
         &= p_B(\BFG{\iota} \cdot \BF{b}) \\
         &= \binom{\frac{\BFG{\iota} \cdot \BF{b}}{\beta} + n -d}{n-d}
     \end{align*}
     and so Eq.~\eqref{eq:binom-dot-product} holds. Since $\BF{b} \in \gamma$, by Proposition \ref{prop:construct-external}, $\BF{b}$ is in the simplicial cone of $A$, $\pos(A_s)$, where $s~=~\{1, \dots, d-1, d+\ell\}$. Therefore, $\BF{b} =  \lambda_1\BF{a}_1 + \dots + \lambda_{d-1}\BF{a}_{d-1} + \lambda_d\BF{a}_{d+\ell}$ for some $\lambda_1, \dots, \lambda_{d} \geq 0$. Then
     \begin{align}
         \BFG{\iota} \cdot \BF{b} &= \lambda_d(\BFG{\iota} \cdot \BF{a}_{d+\ell})  \\
         &= \lambda_d \\
         &= \frac{\det(\BF{a}_1, \dots, \BF{a}_{d-1}, \BF{b})}{\det(\BF{a}_1, \dots, \BF{a}_{d-1}, \BF{a}_{d+\ell})} \label{eq:dot-product-determinant}
     \end{align}
     where the last equality follows from Cramer's rule.
     Eq.~\eqref{eq:binom-determinant} now follows by plugging in Eq.~\eqref{eq:dot-product-determinant} into Eq.~\eqref{eq:binom-dot-product}. 
 \end{proof}

 We note also that in the previous result the columns $\BF{a}_1, \dots, \BF{a}_{d-1}$ can be replaced by any ray generators $\BF{v}_1, \dots, \BF{v}_{d-1}$ with $\pos(\BF{v}_i) = \pos(\BF{a}_i)$ for $i=1,\dots,d-1$. 

 \begin{rem}
    If a column of $A$ is an internal ray generator $\BF{v}$ of $\gamma$ (equivalently some column of $A$ is in $\gamma$ but is not an external column of $\gamma$), then 
    \begin{equation*}
    p_A(\BF{v}) = n- d + 1.
    \end{equation*}
    This is exactly the number of simplicial cones of $A$ that contain $\gamma$ as a subset since there are $n-(d-1)$ choices of $d^{th}$ column to add to the $d-1$ columns on the external facet. For each such simplicial cone $\pos(A_s)$ (i.e with $\gamma \subseteq \pos(A_s)$), there is exactly one solution $\BF{x} \in \mathbb{N}^n$ to $A\BF{x} = \BF{v}$ with $x_i = 0$ for all $i \notin s$ (since the columns of $A$ in $s$ form a basis of $\mathbb{R}^d$ and $\BF{v} \in \pos(A_s)$). Therefore, we find that each solution to $A\BF{x} = \BF{v}$ is of this form, and that no other solutions $\BF{x} \in \mathbb{N}^n$ exist.
\end{rem}
 
We suspect that the saturation condition in the previous theorem can be removed (intuitively we view saturation as ``nice'' from the periodic point of view, so we expect that removing this property on the external columns introduces periodicity).

\subsection{Unimodularity} \label{subsec:unimodularity}

We now consider the case of unimodular matrices -- these matrices have the special property that $p_A$ is a piecewise polynomial, so $p_A^{\gamma}$ is polynomial for each chamber.

A full rank $d \times n$ matrix $A$ with integer entries is \emph{unimodular} if every  $d \times d$ submatrix of $A$ has determinant $1$, $-1$, or $0$ (see for example \cite[Section 19.1]{Schr86}).  If $A$ is unimodular, then $\semigroup{A} = \pos(A) \cap \mathbb{Z}^d$.

In \cite{DeSt03}, De Loera and Sturmfels introduce a generalization of matrix unimodularity given by a geometrical criterion. Both of these definitions appear in this section, so we distinguish them by refering to the older definition simply as unimodular and the one introduced by De Loera and Sturmfels as DeLS-unimodular. A $d \times n$ matrix $A$ with integer entries is defined to be \emph{DeLS-unimodular} if the polyhedron $\{\BF{x} \in \mathbb{R}^d : A\BF{x} = \BF{b}, \BF{x} \geq \BF{0}\}$ associated to the vector partition function $p_A(\BF{b})$ has only integral vertices whenever $\BF{b}$ is in the lattice  spanned by the columns of $A$. Under these conditions, $p_A$ is piecewise polynomial by the following result of De Loera and Sturmfels. We remark that unimodular matrices are DeLS-unimodular. 

\begin{theo}[De Loera, Sturmfels 2003 \cite{DeSt03}] \label{theo:unimodular}
Let $A$ be a $d \times n$ DeLS-unimodular matrix of rank $d$. Then $p_A$ is a piecewise polynomial of degree $n-d$ on $\lattice(A) \cap \pos(A)$ and is zero everywhere else on $\mathbb{Z}^d\cap \pos(A)$. 
\end{theo}

If $A$ is DeLS-unimodular, then each subset of columns of $A$ generates an affine semigroup that is saturated in $\mathcal{L}(A)$. If $A$ is unimodular then it is also DeLS-unimodular, and therefore if $A$ is a unimodular matrix, then the same holds true. The following corollary now follows immediately.

\begin{cor} \label{cor:det-unimodular}
Let $A$ be a $d \times n$ DeLS-unimodular matrix of rank $d$, and $\gamma$ be an external chamber of $A$ with external columns $\BF{a}_1, \BF{a}_2, \dots, \BF{a}_{d-1}$. Let $\BF{a}_{d+\ell}$ be a column of $A$ for some $\ell~\in~\{0, \dots, n-d\}$. Then
\begin{equation} \label{eq:det-unimodular}
p_A^{\gamma}(\BF{b}) = \binom{
\frac{\det(\BF{a}_1, \dots, \BF{a}_{d-1}, \BF{b})}{\det(\BF{a}_1, \dots, \BF{a}_{d-1}, \BF{a}_{d+\ell})} + n-d}{n-d}
\end{equation}
for all $\BF{b} \in \semigroup{A} \cap \gamma$.
\end{cor}

\begin{eg}
    Consider the following DeLS-unimodular matrix 
    \[
    D = \begin{bmatrix}
    2 & 0 & 0 & 2 & 2 \\
    0 & 2 & 0 & 2 & 0 \\
    0 & 0 & 2 & 0 & 2
    \end{bmatrix}.
    \]
    that we have obtained by multiplying the matrix in the running example of \cite{DeSt03} by two. This multiplication has no effect on the chamber complex (i.e the chamber complex of $D$ is the same as that in their running example). However, 
    \begin{equation*}
    \lattice(D) := \{(b_1, b_2, b_3) \in \mathbb{Z}^3 : b_1, b_2, b_3 \equiv 0 \mod 2\}
    \end{equation*}
    in our example, whereas in their running example, the lattice spanned by the matrix is $\mathbb{Z}^3$.
    
    The first three columns are external and the other two are not. 
    Additionally, the chamber
    \[
    \gamma := \pos\left(\begin{bmatrix}
        0 \\
        0 \\
        1
    \end{bmatrix},
    \begin{bmatrix}
        0 \\
        1 \\
        0
    \end{bmatrix},
    \begin{bmatrix}
        1 \\
        1 \\
        1
    \end{bmatrix}\right)
    \]
    is external with minimal internal ray generator 
    \[
    \begin{bmatrix}
        1 \\
        1 \\
        1
    \end{bmatrix}.
    \]
    We now compute the polynomial associated to $\gamma$ using Corollarly \ref{cor:det-unimodular} with the first column of $D$ playing the role of $\BF{a}_{d+\ell}$. Let $d$ denote the ratio of determinants -- that is,
    \begin{equation*}
    d := \det\left(\begin{bmatrix}
        0 & 0 & b_1 \\
        0 & 2 & b_2 \\
        2 & 0 & b_3
    \end{bmatrix}\right)\Bigg/
    \det\left(\begin{bmatrix}
        0 & 0 & 2 \\
        0 & 2 & 0 \\
        2 & 0 & 0
    \end{bmatrix}\right)
    = \frac{b_1}{2}.
    \end{equation*}
    Then
    \begin{align*}
    p_D^{\gamma}(\BF{b}) &= \binom{
    d
    +2}{2} \\
    &= \binom{\frac{b_1}{2} + 2}{2}
    \end{align*}
    for all $\BF{b} \in \semigroup{D} \cap \gamma$. Since $b_1 \equiv 0 \mod 2$, the resulting polynomial does indeed yield integers. 
    Finally, remark that we could have also used the fourth or fifth columns of $D$ in the place of the first column (only the external columns of $\gamma$ cannot be used). 
\end{eg}

In the case that $A$ is unimodular (not just DeLS-unimodular), we can further simplify the expression given in Theorem \ref{theo:unimodular}. We begin with the following useful lemma.

\begin{lem} \label{lem:column-dot-product}
    Let $A$ be a $d \times n$ unimodular matrix of rank $d$. Let $f$ be a facet of $\pos(A)$ with minimal inner facet normal $\BFG{\iota}$. Let $\BF{c}$ be a column of $A$. Then 
    \begin{equation*}
        \BFG{\iota} \cdot \BF{c} = \begin{cases}
            0 \text{ if } \BF{c} \in f \\
            1 \text{ if } \BF{c} \notin f.
        \end{cases}
    \end{equation*}
\end{lem}

\begin{proof}
    If $\BF{c} \in f$ then $\BFG{\iota} \cdot \BF{c} = 0$, so assume that $\BF{c} \notin f$. In this case, there are linearly independent columns $\BF{a}_1, \dots, \BF{a}_{d-1}$ of $A$ lying on $f$ so that $\pos(\BF{a}_1, \dots, \BF{a}_{d-1}, \BF{c})$ is a simplicial cone of $A$. Since $A$ is unimodular, the matrix $M$ whose columns are $\BF{a}_1, \dots, \BF{a}_{d-1}, \BF{c}$ must have determinant $\pm 1$. Thus $M$ is invertible over $\mathbb{Z}$, and in particular, there exists a vector $\BF{v}$ in $\mathbb{Z}^d$ such that $\BF{c} \cdot \BF{v} = 1$. Since $\BF{c}, \BF{v} \in \mathbb{Z}^d$ with $\BF{c} \cdot \BF{v} = 1$, we see that $\operatorname{gcd}(c_1, \dots, c_d) = 1$. Finally since $\BFG{\iota} \in \mathbb{Z}^d$ as well, $\BFG{\iota} \cdot \BF{c}$ is integral, and since $\BFG{\iota}$ is minimal, $\operatorname{gcd}(\iota_1, \dots, \iota_d) = 1$, and so $\BFG{\iota} \cdot \BF{c} = 1$ as required.
\end{proof}

\begin{cor} \label{cor:unimodular}
Let $A$ be a $d \times n$ unimodular matrix of rank $d$, $f$ be a facet of $A$ containing exactly $d-1$ columns of $A$, and $\BFG{\iota}$ be the minimal inner normal of $f$. Moreover, let $\BF{a}_1, \dots, \BF{a}_{d-1}$ be the external columns of $A$ on $f$, and let $\gamma$ be the external chamber containing $f$. Then the polynomial $p_A^{\gamma}(\BF{b})$ associated to $\gamma$ is
\begin{align}
    p_A^{\gamma}(\BF{b}) &= \binom{ \BFG{\iota} \cdot \BF{b} + n-d}{n-d}. \label{eq:facet} \\
    &= \binom{|\det(\BF{a}_1, \dots, \BF{a}_{d-1}, \BF{b})| + n-d}{n-d} \label{eq:abs-det} 
\end{align}
\end{cor}

\begin{rem}
We note that $|\det(\BF{a}_1, \dots, \BF{a}_{d-1}, \BF{b})|$ is the continuous volume of the paralleliped 
\begin{equation*}
\Pi := \{\lambda_1\BF{a}_1 + \dots + \lambda_{d-1}\BF{a}_{d-1} + \lambda_d\BF{b} : 0 \leq \lambda_1, \dots, \lambda_d \leq 1\}
\end{equation*}
generated by $\BF{a}_1, \dots, \BF{a}_{d-1}, \BF{b}$.
\end{rem}

\section{Multigraph enumeration} \label{sec:multigraph}

In this section we consider the problem of enumerating the number of labelled multigraphs with vertices $v_1, \dots, v_m$ and a given sequence $d_1, \dots, d_m$ so that $\deg(v_i) = d_i$ for $1~\leq i~\leq m$. In order that this may be encoded as a vector partition function, we allow multiple edges between any pair of vertices but do not allow loops. Most known results for enumerating the number of graphs or multigraphs with a given degree sequence are asymptotic (see for example \cite{BaHa13, GrMc13, McWaWo02}). We give an exact result for a relatively simple case, that we found by identifying external chambers of the corresponding vector partition functions. We have not seen this result in the literature, nor any attempts to approach this problem via the vector partition function formulation. This is somewhat surprising since it is well known that the enumeration of simple graphs with a given degree sequence can be viewed as counting integer points in polytopes.

For a positive integer $m$, define $M_m(d_1, \dots, d_m)$ to be the number of multigraphs on the vertex set $v_1, \dots, v_m$ with degree sequence $(d_1, \dots, d_m)$. We note that we do not assume that the degree sequence is monotonically decreasing unless explicitly stated.

For each pair of distinct vertices $v_i$ and $v_j$ ($1 \leq i \neq j \leq m$), let $x_{i,j}$ denote the number of edges joining $v_i$ and $v_j.$
A multigraph on the vertex set $v_1, \dots, v_m$ has degree sequence $(d_1, \dots, d_m)$ if the following $m$ linear equations are satisfied:
\begin{equation} \label{eq:multigraph-lin-eqs}
    \sum_{\substack{i=1 \\ i \neq j}}^{m} x_{i,j} = d_j \quad \text{for all } j=1, \dots, m.
\end{equation}
The number of edges between any pair of vertices is a non-negative integer, and so one can describe $M_m(d_1, \dots, d_m)$ as the number of solutions $\BF{x} = (x_{1,2}, \dots, x_{m-1,m}) \in \mathbb{N}^m$ satisfying the linear equations of \eqref{eq:multigraph-lin-eqs}. This description leads to the following vector partition function formulation.

\begin{prop} \label{prop:multigraph-vpf}
Let $m$ be a positive integer, and let $G_m$ denote the incidence matrix of the complete graph $K_m$. Then 
\begin{equation*}
M_m(d_1, \dots, d_m) = p_{G_m}(d_1, \dots, d_m)
\end{equation*}
for each degree sequence $(d_1, \dots, d_m) \in \mathbb{N}^m$ on the lattice $d_1 + d_2 + \dots + d_m \equiv 0 \mod 2$. 
\end{prop}

\begin{eg} \label{eg:multigraphs}
Let us compute $M_4(5,4,3,2)$, the number of multigraphs on the vertex set $\{v_1, v_2, v_3, v_4\}$ and degree sequence $\BF{d} = (5,4,3,2).$ By Proposition \ref{prop:multigraph-vpf}, this is equivalent to computing $p_{G_6}(\BF{d})$, and thus of enumerating the number of integer solutions $\BF{x} = (x_{1,2}, x_{1,3}, x_{1,4}, x_{2,3}, x_{2,4}, x_{3,4}) \in \mathbb{Z}^6$ in the polytope $G_4\BF{x} = \BF{d}, \BF{x} \geq \BF{0}$.
Using \emph{Latte} we compute the solutions explicitly; there are six of them. In Table \ref{tab:multigraph} we give each of the solutions $\BF{x} \in \mathbb{Z}^6$ along with the corresponding multigraph. 
\begin{table} 
    \centering
    \begin{tabular}{ccccc}
       \includegraphics[scale=0.55] {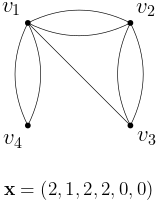} 
       & \quad \quad \quad \quad &
        \includegraphics[scale=0.55]{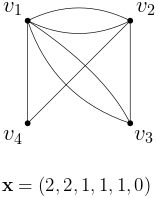}
        & \quad \quad \quad \quad &
        \includegraphics[scale=0.55]{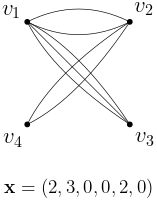} \\
        & & & &  \\
        & & & &  \\
        \includegraphics[scale=0.55] {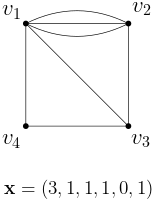} 
       & \quad \quad \quad \quad &
        \includegraphics[scale=0.55]{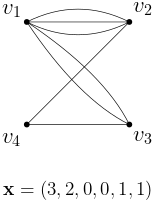}
        & \quad \quad \quad \quad &
        \includegraphics[scale=0.55]{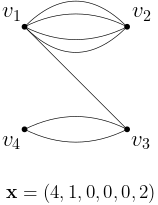}
    \end{tabular}
\caption[The multigraphs with vertices $v_1, v_2, v_3, v_4$ and degree sequence (5,4,3,2).]{The multigraphs with vertices $v_1, v_2, v_3, v_4$ and degree sequence $\BF{d} = (5,4,3,2)$. Each multigraph is labeled by the corresponding solution $\BF{x} \in \mathbb{N}^6$ of $G_4\BF{x} = \BF{d}$.}
\label{tab:multigraph}
\end{table}
\end{eg}

Our goal is to study a particular external chamber of the vector partition function $p_{G_m}$. 

\begin{eg} \label{eg:multigraph}
Let $m=6$. Then 
\[
G_6 = \left(\begin{array}{rrrrrrrrrrrrrrr}
1 & 1 & 1 & 1 & 1 & 0 & 0 & 0 & 0 & 0 & 0 & 0 & 0 & 0 & 0 \\
1 & 0 & 0 & 0 & 0 & 1 & 1 & 1 & 1 & 0 & 0 & 0 & 0 & 0 & 0 \\
0 & 1 & 0 & 0 & 0 & 1 & 0 & 0 & 0 & 1 & 1 & 1 & 0 & 0 & 0 \\
0 & 0 & 1 & 0 & 0 & 0 & 1 & 0 & 0 & 1 & 0 & 0 & 1 & 1 & 0 \\
0 & 0 & 0 & 1 & 0 & 0 & 0 & 1 & 0 & 0 & 1 & 0 & 1 & 0 & 1 \\
0 & 0 & 0 & 0 & 1 & 0 & 0 & 0 & 1 & 0 & 0 & 1 & 0 & 1 & 1
\end{array}\right)
\]
and the chamber $\gamma$ of $G_6$ defined by minimal ray generators
\begin{align*}
\BF{v}_1 := (3, 1, 1, 1, 1, 1), \
\BF{v}_2 := (1, 1, 0, 0, 0, 0), \
\BF{v}_3 := (1, 0, 1, 0, 0, 0), \\
\BF{v}_4 :=(1, 0, 0, 1, 0, 0), \
\BF{v}_5 := (1, 0, 0, 0, 1, 0), \
\BF{v}_6 := (1, 0, 0, 0, 0, 1)
\end{align*}
is external with $\BF{v}_1$ the sole internal ray generator as $\BF{v}_2, \dots, \BF{v}_6$ are external columns of $G_6$.

Our aim is to compute the quasi-polynomial $p_{G_6}^{\gamma}$. The matrix $G_6$ is not DeLS-unimodular (and therefore also not unimodular), since $\mathbf{1} \in \mathcal{L}(G_6)$, but the polytope defined by $G_6\BF{x} = \mathbf{1}$, $\BF{x} \geq 0$ is not integral -- for example one of its vertices is: 
\begin{equation*}
(1/2, 1/2, 0, 0, 0, 1/2, 0, 0, 0, 0, 0, 0, 1/2, 1/2, 1/2).
\end{equation*}
Therefore we cannot use Corollary \ref{cor:det-unimodular} or Corollary \ref{cor:unimodular}. We illustrate two methods, one using Theorem \ref{theo:ehrhart-det} and another using Theorem \ref{theo:external-chamber-binom}.

We begin by showing that the external columns $\{\BF{v}_2, \dots, \BF{v}_6\}$ generate an affine semigroup that is saturated in $\lattice(G_6)$:
if $\BF{u} \in \pos(\BF{v}_2, \dots, \BF{v}_6) \cap \mathcal{L}(G_6)$, then 
\begin{align*}
    \BF{u} &= \lambda_2\BF{v}_2 + \dots + \lambda_6\BF{v}_6 \\
    &= (\lambda_2 + \dots + \lambda_6, \lambda_2, \dots, \lambda_6)
\end{align*}
for some $\lambda_2, \dots, \lambda_6 \geq 0$. Since $\BF{u} \in \mathcal{L}(G_6)$ each of $\lambda_2, \dots, \lambda_6$ must also be integral, and so $\BF{u} \in \semigroup{G_6}.$ Therefore, $\pos(\BF{v}_2, \dots, \BF{v}_6) \cap \mathcal{L}(G_6) \subseteq \semigroup{G_6}$. The reverse inclusion is immediate, so $\semigroup{\BF{v}_2, \dots, \BF{v}_6}$ is indeed saturated in $\lattice(G_6)$.

Using Latte, we compute that 
\begin{equation*}
h(t) := p_{G_6}(tv_1) = \binom{t + 9}{9},
\end{equation*}
and so by Theorem \ref{theo:ehrhart-det},
\begin{equation*}
p_{G_6}^{\gamma} = h\left(\frac{\det{d, v_2, v_3, v_4, v_5, v_6}}{\det{v_1, v_2, v_3, v_4, v_5, v_6}}\right) = \binom{\frac{-d_1 + d_2 + d_3 + d_4 + d_5 + d_6}{2} + 9}{9}
\end{equation*}
on the lattice $d_1 + d_2 + d_3 + d_4 + d_5 + d_6 \equiv 0 \mod 2$.

Since $p_{G_6}^{\gamma}$ is a polynomial, we see that applying Theorem \ref{theo:external-chamber-binom} is a simpler approach, since no computer aid is required. For the external facet of $\gamma$ generated by $\BF{v}_2, \dots, \BF{v}_6$, the minimal inner facet normal is $\BFG{\iota} = (-1, 1, 1, 1, 1, 1)$. We note that $\BFG{\iota} \cdot \BF{g} = 2$ for all columns $\BF{g}$ of $G_6$ not on the external facet. Therefore,by Theorem \ref{theo:external-chamber-binom},
\begin{align*}
    p_{G_6}^{\gamma}(\BF{d}) &= \binom{\frac{\BFG{\iota} \cdot \BF{d}}{2} + 9}{9} \\
    &= \binom{\frac{-d_1 + d_2 + d_3 + d_4 + d_5 + d_6}{2} + 9}{9}.
\end{align*}
\end{eg}

We prove combinatorially that this result is true in general. 

\begin{theo} \label{theo:multigraph}
Let $m$ be a positive integer, and let $(d_1, \dots, d_m) \in \mathbb{N}^{m}$ be monotonically decreasing. 
If $d_1 + d_m \geq \sum_{i=2}^{m-1} d_i$, then
\begin{equation} \label{eq:multigraph}
M_m(d_1, \dots, d_m) = \binom{e - d_1 + \binom{m-1}{2} - 1}{{\binom{m-1}{2}} - 1}
\end{equation}
where $e$ is the number of edges $\sum_{i=1}^{m} \frac{d_i}{2}$ of any multigraph with degree sequence $(d_1, \dots, d_m)$.
\end{theo}

\begin{proof}
Since $d_1 + d_m \geq \sum_{i=2}^{m-1} d_i$, we see that 
$2(d_1 + d_m) \geq \sum_{i=1}^{m} d_i$, so $d_1 + d_m \geq \frac{\sum_{i=1}^m d_i}{2}$ and so $d_m \geq e - d_1$. Now consider distributing edges between the vertices $v_2, v_3, \dots, v_m$. There are $e - d_1$ edges to distribute, and $\binom{m-1}{2}$ vertex pairs. Since $d_i \geq d_m \geq  e - d_1$ for each $2 \leq i \leq m$, we may distribute these edges in any way possible, and no vertex will be incident to too many edges. One can see that the number of such choices is given by Eq.~\eqref{eq:multigraph}. This leaves a single way to distribute the remaining edges from $v_1$ to $\{v_2, v_3, \dots, v_m\}$. 
\end{proof}

\begin{rem}
Recall that in Example \ref{eg:multigraphs} we computed $M_4(5,4,3,2)$ using \emph{Latte}. Since $d_1 + d_4 = 5+2 \geq 4+3 = d_2 + d_3$, we can also apply Theorem \ref{theo:multigraph} to compute this value. Here $e = 7, d_1 = 5$ and $m = 4$, and so by Theorem \ref{theo:multigraph}:
\begin{equation*}
M_4(5,4,3,2) = \binom{7-5 + \binom{4-1}{2} - 1}{\binom{4-1}{2}-1} = 6
\end{equation*}
agreeing with the value computed with \emph{Latte}.
\end{rem}

Although Theorem \ref{theo:multigraph} in the end has a simple combinatorial proof, the vector partition function approach yields the correct inequalities to consider for which the formula becomes simple. We also give a geometric proof of Theorem \ref{theo:multigraph}. Recall from Example \ref{eg:multigraph}, that the matrices $G_m$ are not DeLS-unimodular in general, so we cannot use Corollary \ref{cor:det-unimodular} or Corollary \ref{cor:unimodular} to compute the polynomial. Instead, we use Theorem \ref{theo:external-chamber-binom}.

\begin{lem}
Let $m \geq 3$ be an integer. The cone defined by the minimal ray generators
\[
\gamma := \{(m-3)\BF{e}_1 + \sum_{i=2}^{m} \BF{e}_i\} \cup \{\BF{e}_1 + \BF{e}_i : 2 \leq i \leq m\}
\]
is an external chamber of $G_m$. Moreover the ray generator in the first set is the sole internal ray generator, and the ray generators in the second set are external columns. Finally the affine semigroup generated by the second set is saturated in $\lattice(G_m)$.
\end{lem} 

\begin{proof}
The saturation of the second set in $\lattice(G_m)$ follows similarly to the proof in Example~\ref{eg:multigraph}.

We now prove that $\gamma$ is an external chamber of $G_m$.
The vector $\BF{e}_1 + \BF{e}_i$ is a column of $G_m$ corresponding to the edge $(1, i)$ of the complete graph $K_m$. Additionally, every other column $\BF{c}$ must have a non-zero entry $c_j$ for some $j \notin \{1, i\}$. Therefore, $\BF{e}_1 + \BF{e}_i$ is not in the cone generated by the other $\binom{m}{2} - 1$ columns of $G_m$, and is thus an external ray generator. The cone $f$ generated by $\{\BF{e}_1 + \BF{e}_i : 2 \leq i \leq m\}$ is the $(n-1)$-dimensional intersection of the cone $\pos(G_m)$ with the hyperplane $d_1 - \sum_{i=2}^{m} d_i$, and is thus a facet of $\pos(G_m)$. No other columns of $G_m$ appear in the facet $f$, so $f$ is an external facet. By Proposition \ref{prop:construct-external}, the external chamber $\tilde{\gamma}$ containing $f$ is the intersection of all simplicial cones of $G_m$ containing each of the columns $\BF{e}_1 + \BF{e}_i$ for $2 \leq i \leq m$. For $j,k$ with $1 < j < k < m$, set $\sigma_{j,k} := \pos(f, \BF{e}_j + \BF{e}_k)$ where we overload notation by denoting the ray generators of $f$ by $f$. Then 
\begin{equation*}
\tilde{\gamma} = \bigcap_{1 < j < k \leq m} \sigma_{j,k}.
\end{equation*}

We now prove that $\gamma = \tilde{\gamma}$.

The cone $\sigma_{j,k}$ is defined by the inequalities
\begin{align*}
d_i &\geq 0 \text{ for all } i \neq 1,j,k\\
\sum_{i=2}^{m} d_i &\geq d_1 \\
d_1 + d_j &\geq \sum_{i=2, i \neq j}^{m} d_i \\
d_1 + d_k &\geq \sum_{i=2, i \neq k}^{m} d_i 
\end{align*}
and so $\tilde{\gamma}$ is defined by the inequalities
\begin{align}
\sum_{i=2}^{m} d_i &\geq d_1 \label{eq:sum-rest} \\
d_1 + d_l &\geq \sum_{i=2, i \neq l}^{m} d_i  \text{ for all } 2 \leq l \leq m  \label{eq:one-plus}
\end{align}
where the inequalities $d_i \geq 0$ for $2 \leq i \leq m$ follow implicilty from Inequalities~\eqref{eq:sum-rest} and \eqref{eq:one-plus}.
This is the same set of inequalities defining $\gamma$. 
\end{proof}

The inequality defining the external facet $f$ is $\sum_{i=2}^{m} d_i \geq d_1$. Thus we see that for column $k$ corresponding to the vector $\BF{e}_i + \BF{e}_j$ ($1 \leq i < j \leq m$) the entry $(1, k)$ of $M_{\gamma}^{\vee}G_m$ is
\begin{equation*}
(-1, 1, \dots, 1) \cdot (\BF{e}_i + \BF{e}_j) = \begin{cases}
0 \text{ if } i = 1 \\
2 \text{ if } i \neq 1
\end{cases}
\end{equation*}
and so by Theorem \ref{theo:external-chamber-binom},
\begin{equation*}
p_ {G_m}(\BF{d}) = \binom{\frac{-d_1 + d_2 + \dots + d_m}{2} + \binom{m -1}{2} - 1}{\binom{m -1}{2} - 1}
\end{equation*}
on the chamber defined by Inequalities \eqref{eq:sum-rest} and \eqref{eq:one-plus}. 

This concludes our geometric proof of Theorem \ref{theo:multigraph}.

\section{Proof of Lemma \ref{lem:dim-reduction}} \label{sec:lem-proof}
\subsection{Statement and notation} \label{subsec:lem-proof-notation}
We begin by recalling Lemma \ref{lem:dim-reduction}, written here as Lemma \ref{lem:dim-reduction-appendix} below. 

\begin{lem}  \label{lem:dim-reduction-appendix}
Let $\gamma$ be a chamber of $A$. Assume without loss of generality that the external columns of $A$ in $\gamma$ are $\BF{a}_1, \dots, \BF{a}_{\ell}$ for some $\ell \in \{0,\dots, d-1\}$. Also assume that $\BF{a}_i = k_i\BF{e}_i$ for each $i \in \{1, \dots, \ell\}$ and some positive integers $k_1, \dots, k_{\ell}.$ Finally assume that $\semigroup{\BF{a}_1, \dots, \BF{a}_{\ell}}$ is saturated in $\lattice(A)$.
Let $B$ be the matrix obtained by removing the first $\ell$ rows and columns of $A$. 
Then there exists a chamber $\gamma'$ of $B$ such that
\begin{equation} \label{eq:vpf-equivalence} 
p^{\gamma}_A(\BF{b}) = p_{B}^{\gamma'}(b_{\ell+1}, \dots, b_d) 
\end{equation}
for all $\BF{b} = (b_1, \dots, b_n) \in \semigroup{A} \cap \gamma$.
\end{lem}

To prove this lemma, we proceed by induction on $\ell$. The base case $\ell = 0$ is clear. We assume henceforth that $\ell \geq 1$. For the inductive step there are two things that we need to show (labelled [\textsf{C}] and [\textsf{V}] below):
\begin{enumerate}
    \item[{[\textsf{C}]}] The matrix $A_{\hat{1}, \hat{1}}$ also satisfies the hypotheses of Lemma \ref{lem:dim-reduction}, that is:
    \begin{enumerate}
        \item there is a chamber $\gamma'$ of $A_{\hat{1}, \hat{1}}$ which is simplicial. Furthermore, for all $\BF{b} \in \gamma$, it follows that $\BF{b}_{\hat{1}} \in \gamma'$,
        \item the columns $\{(\BF{a}_2)_{\hat{1}}, \dots, (\BF{a}_{\ell})_{\hat{1}}\}$ are external columns of $A_{\hat{1}, \hat{1}}$,
        \item the affine semigroup $\semigroup{(\BF{a}_2)_{\hat{1}}, \dots, (\BF{a}_{\ell})_{\hat{1}}}$ is saturated in $\lattice(A_{\hat{1}, \hat{1}})$.
    \end{enumerate}
    
    \item[{[\textsf{V}]}] The vector partition function of $A_{\hat{1}, \hat{1}}$ respects Eq.~\eqref{eq:vpf-equivalence}, that is: $p_A^{\gamma}(\BF{b}) = p^{\gamma'}_{A_{\hat{1}, \hat{1}}}(\BF{b}_{\hat{1}})$ for all $\BF{b} \in \semigroup{A} \cap \gamma$. \label{item:V}
\end{enumerate}

By proving [\textsf{C}] and [\textsf{V}], we show that we can iteratively remove the $\ell$ rows and columns of $A$ corresponding to the external columns of $\gamma$.

Having outlined our plan, we begin by introducing some notation. 

Throughout this appendix, we assume that $\gamma$ is a simplicial chamber of $A$. We additionally assume that $\gamma$ has external columns $\BF{a}_1, \dots, \BF{a}_{\ell}$ satisfying $\BF{a}_i = k_i\BF{e}_i$ for some positive integers $k_1, \dots, k_{\ell}$, and that the affine semigroup $\semigroup{\BF{a}_1, \dots, \BF{a}_{\ell}}$ is saturated in $\lattice(A)$.

We also use $\tilde{A}$ to denote the matrix $A_{\hat{1}, \hat{1}}$, and we overload notation by denoting the columns of $\tilde{A}$ by $\tilde{\BF{a}}_2, \dots, \tilde{\BF{a}}_n$ (in order to keep the indexing consistent between $A$ and $\tilde{A}$). Similarly for a subset $\tilde{s} \subseteq \{2, \dots, n\}$, we write $\tilde{A}_{\tilde{s}}$ to indicate the submatrix of $\tilde{A}$ whose columns are $\{\tilde{\BF{a}}_i : i \in \tilde{s}\}$. Finally, by $B$, we denote the matrix obtained by removing the first $\ell$ rows and columns of $A$. We write the general form of the matrices $A$ and $\tilde{A}$ below: 

\begin{align*}
A = 
    \begin{blockarray}{ccccccc}
        \BF{a}_1 & \BF{a}_2 & \dots & \BF{a}_{\ell} & \BF{a}_{\ell+1} & \dots & \BF{a}_n \\
        \begin{block}{(cccc|ccc)}
		k_1 &  & &  & \BAmulticolumn{3}{c}{\multirow{4}{*}{\Huge $*$}}   \\
		  & k_2 &  &  \\
            & & \ddots & \\
            & &  & k_{\ell}  \\ 
            \cline{1-7} 
            \BAmulticolumn{4}{c|}{\multirow{3}{*}{\Large 0}} & \BAmulticolumn{3}{c}{\multirow{3}{*}{\Large B}} \\
            & & & & & &  \\
            & & & & & &  \\
        \end{block} 
    \end{blockarray} 
& \ & \ &
\tilde{A} = 
    \begin{blockarray}{cccccc}
        \tilde{\BF{a}}_2 & \dots & \tilde{\BF{a}}_{\ell} & \tilde{\BF{a}}_{\ell+1} & \dots & \tilde{\BF{a}}_n \\
        \begin{block}{(ccc|ccc)}
		k_2 &  &  & \BAmulticolumn{3}{c}{\multirow{3}{*}{\Huge $*$}}\\
              & \ddots & \\
             &   & k_{\ell}  \\ 
            \cline{1-6} 
            \BAmulticolumn{3}{c|}{\multirow{3}{*}{\Large 0}} & \BAmulticolumn{3}{c}{\multirow{3}{*}{\Large B}} \\
            & & & & &  \\
            & & & & &  \\
        \end{block} 
    \end{blockarray}. 
\end{align*}

\subsection{Proof of [{\textsf{C}]}} \label{subsec:proof-of-[c]}
\textbf{The simplicial cones of \texorpdfstring{$\tilde{A}$}{A tilde}}

We now prove some results which relate the simplicial cones of $A$ with those of $\tilde{A}$ with a view towards proving the conditions of [\textsf{C}]. 

Recall that $\pos(A_s)$ is a simplicial cone of $A$ if and only if $|s| = \rank(A_s) = d$. 

\begin{prop} \label{prop:simplicial-simplicial}
    Let $\tilde{s} \subseteq \{2, \dots, n\}$, and let $s := \{1\} \cup \tilde{s}$. Then $\pos(A_s)$ is a simplicial cone of $A$ if and only if $\pos(\tilde{A}_{\tilde{s}})$ is a simplicial cone of $\tilde{A}$.
\end{prop}

\begin{proof}
    This result follows immediately from the observation that $\rank(A_s) = \rank(\tilde{A}_{\tilde{s}})~+~1$.
\end{proof}

In the following proposition we use $\tilde{\BF{b}}$ to denote $\BF{b}_{\hat{1}}$.

\begin{prop} \label{prop:simplicial-b} 
    Let $\BF{b} \in \gamma$. Let $\pos(\tilde{A}_{\tilde{s}})$ be a simplicial cone of $\tilde{A}$ for some subset $\tilde{s}~\subseteq~\{2,\dots,n\}$, and let $s := \{1\} \cup \tilde{s}$. Then $\BF{b} \in \pos(A_s)$ if and only if $\tilde{\BF{b}} \in \pos(\tilde{A}_{\tilde{s}})$. 
\end{prop}

\begin{proof}
    The forward direction (if $\BF{b} \in \pos(A_s)$ then $\tilde{\BF{b}} \in \pos(\tilde{A}_{\tilde{s}})$) is clear, so we prove only the reverse direction.

    Since  $\tilde{\BF{b}} \in \pos(\tilde{A}_{\tilde{s}})$, 
    \begin{equation*}
        \tilde{\BF{b}} = \sum_{i \in \tilde{s}} \lambda_i \tilde{\BF{a}}_i
    \end{equation*}
    for some $\lambda_i \geq 0$. Let 
    \begin{equation*}
        \BF{c} :=  \sum_{i \in \tilde{s}} \lambda_i\BF{a}_i.
    \end{equation*}
    There are two cases to consider.

    \textit{Case 1:} If $c_1 \leq b_1$, then 
    \begin{align*}
        \BF{b} &= (b_1 - c_1)\BF{e}_1 + \BF{c} \\
        &= \frac{b_1 - c_1}{k_1}\BF{a}_1 + \sum_{i \in \tilde{s}} \lambda_i\BF{a}_i \\
        &= \sum_{i \in s}\lambda_i\BF{a}_i & (\text{setting } \lambda_1 := \frac{b_1 - c_1}{k_1})
    \end{align*}
    and so $\BF{b} \in \pos{A_s}$.
    
    \textit{Case 2:} If $c_1 > b_1$, then
    \begin{align*}
        \BF{c} &= \BF{b} + (c_1 - b_1)\BF{e}_1 \\
        &= \BF{b} + \frac{c_1 - b_1}{k_1}\BF{a}_1
    \end{align*}
    and since $\BF{a}_1, \BF{b} \in \gamma$, it follows that $\BF{c} \in \gamma$. Moreover, by definition $\BF{c} \in \pos(A_{\hat{1}})$. Therefore, $\BF{c}$ lies on the unique facet $f$ of $\gamma$ not containing $\BF{a}_1$. Since $f$ is a face of $\gamma$, if the sum of any two vectors in $\gamma$ is in $f$, then both of those vectors must be in $f$ (see \cite[Section 1.2]{Fu93} for example). Therefore, $\BF{b}, \frac{c_1 - b_1}{k_1}\BF{a}_1 \in f$ and in particular, $\BF{a}_1 \in f$. This is a contradiction since $f$ is the unique facet of $\gamma$ not containing $\BF{a}_1$, and so Case 2 cannot occur.

    Therefore, $\BF{b} \in \pos(A_s)$ as required.
\end{proof}

\paragraph{(a) Simplicial chamber:}

Let $S$ denote the subset of $\mathcal{P}(\{1, 2, \dots, n\})$ such that $s \in S$ if and only if $\pos(A_s)$ is a simplicial cone satisfying $\gamma \subseteq \pos(A_s)$, and let $\tilde{S} := \{s \setminus \{1\} : s \in S\}.$

\begin{prop} \label{prop:chambers-for-smaller}
    The cone 
    \begin{equation} \label{eq:chamber-A-tilde}
        \tilde{\gamma} := \bigcap_{\tilde{s} \in \tilde{S}} \pos(\tilde{A}_{\tilde{s}})
    \end{equation}
    is a chamber of $\tilde{A}$. Additionally, $\BF{b} \in \gamma$ if and only if $\tilde{\BF{b}} \in \tilde{\gamma}$.
\end{prop}

\begin{proof}
    Recall that we can represent $\gamma$ as the intersection of all simplicial cones of $A$ containing $\BF{b}$ for some $\BF{b} \in \gamma^{\circ}$. Consider $\tilde{\BF{b}} := \BF{b}_{\hat{1}}$. By Proposition \ref{prop:simplicial-b}, it follows that the simplicial cones of $\tilde{A}$ containing $\tilde{\BF{b}}$ are exactly the simplicial cones appearing on the right-hand side of \eqref{eq:chamber-A-tilde}. Therefore, $\tilde{\gamma}$ is a cone in the chamber complex of $\tilde{A}$. By construction, $b \in \gamma$ if and only if $\tilde{b} \in \tilde{\gamma}.$

    In order to prove that $\tilde{\gamma}$ is a chamber of $\tilde{A}$, we must show that it is $(d-1)$-dimensional. Let $G$ be the matrix whose columns are the minimal ray generators of $\gamma$ (up to column permutation) so that the first column of $G$ is $\BF{a}_1 = k_1\BF{e}_1$. Then, the columns of $G_{\hat{1}, \hat{1}}$ are exactly the minimal ray generators of $\tilde{\gamma}$. Since $\rank(G) = \rank(G_{\hat{1}, \hat{1}}) + 1$, it follows that  $\tilde{\gamma}$ is indeed $(d-1)$-dimensional, and thus a chamber of $\tilde{A}.$
\end{proof}

We note that $\tilde{\BF{a}}_2, \dots, \tilde{\BF{a}}_{\ell} \in \tilde{\gamma}$. In the next section we show that each of these columns is an external column of $\tilde{A}.$

\paragraph{(b) External columns:}

\begin{prop}
    Each of the columns $\tilde{\BF{a}}_2, \dots, \tilde{\BF{a}}_{\ell}$ are external columns of $\tilde{A}$.
\end{prop}

\begin{proof}
    We give the proof for $\tilde{\BF{a}}_2$ noting that the other cases follow similarly. 

   Assume towards a contradiction that $\tilde{\BF{a}}_2$ is not an external column of $\tilde{A}$. Since a cone can be triangulated into simplicial cones with no new ray generators, there is some subset $s' \subseteq \{3, \dots, n\}$ such that
    \begin{equation*}
    \tilde{\BF{a}}_2 \in \pos(\tilde{A}_{s'})
    \end{equation*}
    and $\{\tilde{\BF{a}_i} : i \in s'\}$ is linearly independent.
    Additionally, since $\tilde{\BF{a}}_2$ is part of a linearly dependent set, we see that $s'$ can be extended to some set $\tilde{s} \subseteq \{3, \dots, n\}$ with $\rank(\tilde{A}_{\tilde{s}}) = |\tilde{s}| = d-1$. Therefore, $\tilde{\BF{a}}_2$ is in the simplicial cone $\pos(\tilde{A}_{\tilde{s}})$ of $\tilde{A}$, and so by Proposition \ref{prop:simplicial-b}, $\BF{a}_2 \in \pos(A_s)$ where $s = \tilde{s} \cup \{1\}$. This is a contradiction since $\BF{a}_2$ is an external column. 
\end{proof}

\paragraph{(c) Saturation:}

Recall that the semigroup $\semigroup{A_s}$ is saturated in $\lattice(A)$ for some $s \subseteq \{1, 2, \dots, n\}$ if 
    \begin{equation*}
        \semigroup{A_s} = \lattice(A) \cap \pos(A_s).
    \end{equation*}
We note again that $\semigroup{\BF{a}_1, \dots, \BF{a}_{\ell}}  \subseteq  \lattice(A) \cap \pos(\BF{a}_1, \dots, \BF{a}_{\ell})$ for each choice of columns, so one only needs to show the reverse inclusion. 

\begin{prop}
    The affine semigroup $\semigroup{\tilde{\BF{a}}_2, \dots, \tilde{\BF{a}}_{\ell}}$ is saturated in $\lattice(\tilde{A})$.
\end{prop}

\begin{proof}
    We need to prove that $\pos(\tilde{\BF{a}}_2, \dots, \tilde{\BF{a}}_{\ell}) \cap \mathcal{L}(\tilde{A}) \subseteq \semigroup{\tilde{\BF{a}}_2, \dots, \tilde{\BF{a}}_{\ell}}$.
    
    Let $\tilde{\BF{b}} \in \pos(\tilde{\BF{a}}_2, \dots, \tilde{\BF{a}}_{\ell}) \cap \mathcal{L}(\tilde{A}).$ Since $\tilde{\BF{b}} \in \mathcal{L}(\tilde{A})$, there exist integers $d_2, \dots, d_n$ such that
    \begin{equation*}
        \tilde{\BF{b}} = \sum_{i=2}^{n} d_i \tilde{\BF{a}_i}.
    \end{equation*}
    and since $\tilde{\BF{b}} \in \pos(\tilde{\BF{a}}_2, \dots, \tilde{\BF{a}}_{\ell})$, there exist $\lambda_2, \dots, \lambda_n \geq 0$ such that
    \begin{equation*}
        \tilde{\BF{b}} = \sum_{i=2}^{n} \lambda_i \tilde{\BF{a}_i}.
    \end{equation*}
    Let 
    \begin{equation*}
        \BF{c} := \sum_{i=2}^{n} d_i \BF{a}_i \in \lattice(A),
    \end{equation*}
    and let 
    \begin{equation*}
        \BF{c}' := \sum_{i=2}^{n} \lambda_i \BF{a}_i \in \pos(\BF{a}_1, \dots, \BF{a}_{\ell}).
    \end{equation*}

    By construction, $c_2 = c'_2, \dots, c_d = c'_d.$
    Let $N$ be a positive integer satisfying $k_1N > c'_1 - c_1$, and let $\BF{b} := \BF{c} + N\BF{a}_1$, so that $b_1 = c_1 + k_1N$. Then $\BF{b} \in \lattice(A)$ since $\BF{a}_1, \BF{c} \in  \lattice(A)$ and $N$ is an integer. Additionally, 
    \begin{align*}
        \BF{b} &= \BF{c}' + (k_1N + c_1 -  c'_1)\BF{e}_1  \\
        &= \BF{c}' + \frac{k_1N + c_1 - c'_1}{k_1}\BF{a}_1
    \end{align*}
    and since $\BF{a}_1, \BF{c}' \in \pos(\BF{a}_1, \dots, \BF{a}_{\ell})$ and $\frac{k_1N + c_1 - c'_1}{k_1} > 0$, it follows that $\BF{b} \in \pos(\BF{a}_1, \dots, \BF{a}_{\ell})$.

    Since $\semigroup{\BF{a}_1, \dots, \BF{a}_{\ell}}$ is saturated in $\lattice(A)$ it follows that $\BF{b} \in \semigroup{\BF{a}_1, \dots, \BF{a}_{\ell}}$, and so 
    \begin{equation*}
        \BF{b} = m_1\BF{a}_1 + \dots + m_{\ell}\BF{a}_{\ell}
    \end{equation*}
    for some non-negative integers $m_1, \dots, m_{\ell}$. Therefore, $\tilde{\BF{b}} = m_2\tilde{\BF{a}_2} + \dots + m_{\ell}\tilde{\BF{a}}_{\ell},$ and thus $\tilde{\BF{b}} \in \pos(\tilde{\BF{a}}_2, \dots, \tilde{\BF{a}_{\ell}})$ as required.
\end{proof} 

\subsection{Proof of [{\textsf{V}]}} \label{subsec:proof-of-[v]}
Let $\tilde{\gamma}$ be defined as in \eqref{eq:chamber-A-tilde} -- equivalently
\begin{equation*}
    \tilde{\gamma} = \{\BF{b}_{\hat{1}} : \BF{b} \in \gamma\}.
\end{equation*}

Before proceeding to the proof we make a quick remark about saturation which we exploit.
If the semigroup $\semigroup{\BF{a}_1, \dots, \BF{a}_{\ell}}$ is saturated in $\lattice(A)$ and each $\BF{a}_j$ for $j=1, \dots, \ell$ is an external column, then the affine semigroups generated by the singleton sets $\{\BF{a}_j\}$ are also saturated in $\lattice(A)$. 

\begin{lem} \label{lem:chamber-dimension-lower}
Let $\BF{b} \in \semigroup{A} \cap \gamma$, and let $\tilde{\BF{b}} := \BF{b}_{\hat{1}}$. Then
\begin{equation*} 
p_A^{\gamma}(\BF{b}) = p^{\tilde{\gamma}}_{\tilde{A}}(\tilde{\BF{b}}).
\end{equation*}
\end{lem}

\begin{proof}
    The following proof proceeds in two stages. In the first stage, we use the assumption that $\BF{b} \in \semigroup{A}$. In the second stage we use geometrical methods exploiting the fact that $\BF{b} \in \gamma$.

Denote the entry at the $r$th row and $c$th column of $A$ by $a_{r,c}$. We note that $a_{2,1} \dots, a_{n,1} = 0$, because the first column is $k_1\BF{e}_1$.

Let $\BF{b} = (b_1, \dots, b_d) \in \semigroup{A} \cap \gamma$. Consider some $\tilde{\BF{x}} = (x_2, \dots, x_n) \in \mathbb{N}^{n-1}$ satisfying 
\begin{align}
    &a_{1,2}x_2 + \dots + a_{1,n}x_n \leq b_1 \label{eq:the-ineq} \\
    &a_{2,2}x_2 + \dots + a_{2,n}x_n = b_2 \label{eq:second-eq}\\
    &\hspace{0.4cm} \vdots \hspace{2.1cm} \vdots \hspace{1cm} \vdots \notag \\
    &a_{d,2}x_2 + \dots + a_{d,n}x_n = b_d. \label{eq:last-eq}
\end{align}
For such $\tilde{\BF{x}}$, 
\begin{align*}
    \BF{u} &:= \BF{b} - x_2\BF{a}_2 - \ldots - x_n\BF{a}_n \\
    &= ((\underbrace{b_1 - (a_{1,2}x_2 + \dots + a_{1,n}x_n}_{m}), 0, \dots, 0) \\
    &= m\BF{e}_1
\end{align*}
where $m$ is a non-negative integer. So, $\BF{u} \in \pos(\BF{e}_1)$ and since $\BF{b} \in \semigroup{A} \subseteq \lattice(A)$, and $\BF{a}_2, \dots, \BF{a}_n \in \lattice(A)$, it follows that $\BF{u} \in \lattice(A)$ as well. By hypothesis 
\begin{align*}
    \lattice(A) \cap \pos(\BF{e}_1) &= \semigroup{\BF{a}_1} \\
    &= \semigroup{k_1\BF{e}_1} 
\end{align*}
hence $k_1 | m$. 
Any such $\tilde{\BF{x}}$ extends uniquely to the solution 
\begin{equation*}
    \BF{x} = \left(\frac{m}{k_1}, x_2, \dots, x_n\right) \in \mathbb{N}
\end{equation*}
of $A\BF{x} = \BF{b}$. We note that this is the unique choice for $x_1$, because
\[
A\BF{x} = \begin{bmatrix}
    k_1x_1 + b_1 - m \\
    b_2 \\
    \vdots \\
    b_d
\end{bmatrix}
\]
and so $k_1x_1 + b_1 - m = b_1$ implying that $x_1 = m/k_1$.
 The previous argument describes an injective map from the set of solutions $\BF{x'} \in \mathbb{N}^{d-1}$ satisfying Lines \eqref{eq:the-ineq}--\eqref{eq:last-eq} to the set of solutions $\BF{x} \in \mathbb{N}^d$ of $A\BF{x} = \BF{b}$. This map is clearly a bijection since the inverse map is the projection $\BF{x} = (x_1, \dots, x_n) \mapsto (x_2, \dots, x_n)$. 

 We now show that the conditions $\BF{b} \in \gamma$ and $A_{\hat{1}, \cdot}\BF{x} = \BF{b}_{\hat{1}}$ are sufficient to imply Inequality~\eqref{eq:the-ineq}. The second of these assumptions will be used to obtain Eq.~\eqref{eq:other-eqs-hold} from Eq.~\eqref{eq:dot-prod}. Equations \eqref{eq:second-eq} -- \eqref{eq:last-eq} are equivalent to $A_{\hat{1}, \cdot}\BF{x} = \BF{b}_{\hat{1}}$ since $a_{2,1}, \dots, a_{n,1} = 0$.

By Lemma \ref{lem:unique-facet}, $\gamma$ has a unique facet $f = \gamma \cap H$ not containing $k_1\BF{e}_1$, where $H$ is a supporting hyperplane of $\gamma$ separating $\BF{a}_1$ from $\pos(A_{\cdot, \hat{1}})$. Let  $\BFG{\iota}$ denote an inner normal of $H$ with respect to the cone $\gamma$. Since $k_1\BF{e}_1 \in H^+$, we have that $\BFG{\iota} \cdot (k_1\BF{e}_1) > 0$ and so $\iota_1 > 0$. By definition, we also have that $\BFG{\iota}^T\BF{b} \geq 0$ for all $\BF{b} \in \gamma$, and since $H$ separates $\BF{a}_1$ from $\pos(A_{\cdot, \hat{1}})$, we find that $\pos(A_{\cdot, \hat{1}}) \subseteq H^-$ and so $\BFG{\iota}^TA_{\cdot, \hat{1}} \BF{x}_{\hat{1}, \cdot}\leq 0$. 
 
 We have 
\begin{align}
\BFG{\iota}^T\BF{b} &=  \iota_1b_1 + \BFG{\iota}_{\hat{1}} \cdot \BF{b}_{\hat{1}} \label{eq:dot-prod} \\
&= \iota_1b_1 + \BFG{\iota}_{\hat{1}}^TA_{\hat{1},\cdot}\BF{x} \label{eq:other-eqs-hold} \\
&= \iota_1b_1 + \BFG{\iota}^TA_{\cdot, \hat{1}}x_{\hat{1}} - (\iota_1a_{1,2}x_2 + \dots + \iota_1a_{1,n}x_n) \label{eq:introduce-Ahat} \\
&=\iota_1\biggr(b_1 - (a_{1,2}x_2 + \dots + a_{1,n}x_n)\biggr) + \BFG{\iota}^TA_{\cdot, \hat{1}}\BF{x}_{\hat{1}} \label{eq:with-b1}
\end{align} 
where Eq.\eqref{eq:introduce-Ahat} follows since $a_{2,1} = \dots = a_{d,1} = 0$ as the first column of $A$ is $k_1\BF{e}_1$. Rearranging Eq.~\eqref{eq:with-b1} to solve for $b_1$, we have
\begin{align*}
b_1  &= \frac{\BFG{\iota}^T\BF{b} - \BFG{\iota}^TA_{\cdot, \hat{1}}\BF{x}_{\hat{1}}}{\iota_1} + (a_{1,2}x_2 + \dots + a_{1,n}x_n) \\
&\geq a_{1,2}x_2 + \dots + a_{1,n}x_n.
\end{align*}

Therefore, we see that Equations \eqref{eq:second-eq} -- \eqref{eq:last-eq} do indeed imply Inequality \eqref{eq:the-ineq}. Since $\BF{x'}$ satisfying Inequality \eqref{eq:the-ineq} and Equations \eqref{eq:second-eq} -- \eqref{eq:last-eq} extends to a solution $\BF{x}$ of $A\BF{x} = \BF{b}$, we find that for $\BF{b} \in \gamma \cap \semigroup{A}$, the sets $\{\BF{x} \in \mathbb{N}^{n} : A\BF{x} = \BF{b}\}$ and $\{\BF{x} \in \mathbb{N}^{n} : A_{\hat{1}, \cdot}\BF{x} = \BF{b}_{\hat{1}}\}$ are equal. Therefore, $p_A(\BF{b}) = p_{A_{\hat{1}, \cdot}}(\BF{b}_{\hat{1}})$. Finally since the first column of $A_{\hat{1}, \cdot}$ is all zeroes, we have that $p_{A_{\hat{1}, \cdot}}(\BF{b}_{\hat{1}}) = p_{A_{\hat{1}, \hat{1}}}(\BF{b}_{\hat{1}})$, and so 
\[
p_A(\BF{b}) = p_{A_{\hat{1}, \hat{1}}}(\BF{b}_{\hat{1}}) = p^{\tilde{\gamma}}_{\tilde{A}}(\tilde{\BF{b}})
\]
completing the proof.
\end{proof}

Now that we have proven [\textsf{C}] and [\textsf{V}], we see that we can indeed iteratively remove the first $\ell$ rows and columns of $A$. This concludes our proof of Lemma \ref{lem:dim-reduction}.

\section{Semi-external chambers} \label{sec:lin-factors}

In this section we consider a generalization of external chambers -- semi-external chambers. The main motivation for this is to examine the appearance of linear factors for polynomials associated to such chambers (under some additional conditions). Such a result (due to Baldoni and Vergne \cite[Corollary 14]{BaVe08}) exists in the unimodular case, but we have observed that linear factors appear in a more general setting. In Conjecture~\ref{conj:lin-factors} we give a possible generalization (which notably also generalizes Theorem~\ref{theo:external-chamber-binom} in the $\beta = 1$ case). If this conjecture is valid, it would allow us not only to compute linear factors in the multigraph enumeration case described in Section~\ref{sec:multigraph}, but also possibly of Littlewood-Richardson polynomials (see Section~\ref{sec:future} for more details). 

\begin{defn}
For a $d \times n$ matrix $A$ of rank $d$, we define a chamber of $A$ to be \emph{semi-external} if it intersects a facet of $\pos(A)$ $(d-1)$-dimensionally. 
\end{defn}

\begin{eg}
Recall the matrix $K_3$ and its chamber complex given in Example \ref{eg:kostant-eg1}. The cone $\pos(K_3)$ has three facets. They are $f_1 := \pos(\BF{a}_1, \BF{a}_2, \BF{a}_4),\ f_2 := \pos(\BF{a}_1, \BF{a}_6),\ f_3 := \pos(\BF{a}_4, \BF{a}_5, \BF{a}_6)$. 
Chambers $\gamma_3$ and $\gamma_6$ are the chambers of $K_3$ that meet $f_1$ $2$-dimensionally, $\gamma_5$ is the chamber that meets $f_2$ $2$-dimensionally, and $\gamma_1$ and $\gamma_4$ are the chambers that meet $f_3$ $2$-dimensionally. Therefore, of the seven chambers of $K_3$, Chambers $\gamma_1, \gamma_3, \gamma_4, \gamma_5, \gamma_6$ are semi-external, and the rest are not. It may help the reader to examine the projection of the chamber complex of $K_3$ given in Figure \ref{fig:kostant-chambers1}. Pictorally, the semi-external chambers here correspond to the $2$-dimensional regions bounded by edges (the projections of the chambers of $K_3$) that intersect the outer triangle (the projection of $\pos(K_3)$) $1$-dimensionally (since we are projecting $2$-dimensional intersections). 
\end{eg}

External chambers of $A$ are exactly the semi-external chambers $A$ that meet a simplicial facet of $\pos(A)$ $(d-1)$-dimensionally. Unlike external chambers, semi-external chambers of $A$ always exist -- in fact, there must be at least as many semi-external chambers as there are facets of $\pos(A).$

\begin{theo}[Baldoni, Vergne, 2008 \cite{BaVe08}] \label{theo:lin-factors}
Let $A$ be a $d \times n$ unimodular matrix of rank~$d$, let $f$ be a facet of $\pos(A)$ with minimal inner facet normal $\BFG{\iota}$, and let $\gamma$ be a semi-external chamber of $A$ intersecting $f$ $(d-1)$-dimensionally. Let $k$ be the number of columns of $A$ not in $f$. Then $p_{A}^{\gamma}(\BF{y})$ has linear factors 
\begin{equation*}
    (\BFG{\iota} \cdot \BF{y}) + i
\end{equation*}
for $i=1,\dots,k-1$. 
\end{theo}

As previously stated, we suspect that Theorem \ref{theo:lin-factors} can be generalized since we have found evidence of linear factors appearing in non-unimodular cases.
In particular, we suspect that the key is the dot product condition of Lemma \ref{lem:column-dot-product}. 
Motivated by this, we make the following conjecture.

\begin{conj} \label{conj:lin-factors}
    Let $A$ be a $d \times n$ matrix of rank $d$ with integer entries, let $f$ be a facet of $\pos(A)$ with inner facet normal $\BFG{\iota}$, and let $\gamma$ be a semi-external chamber of $A$ intersecting $\gamma$ $(d-1)$-dimensionally.
    Assume that $\BFG{\iota} \cdot \BF{c} = 1$ for each column $\BF{c}$ of $A$ not on $f$ and that $p_A^{\gamma}(\BF{y})$ is a polynomial. Let $k$ be the number of columns of $A$ not on $f$. Then $p_{A}^{\gamma}(\BF{y})$ has linear factors 
\begin{equation*}
    (\BFG{\iota} \cdot \BF{y}) + i
\end{equation*}
for $i=1,\dots,k-1$  for all chambers of $A$ intersecting $f$ $(d-1)$-dimensionally.
\end{conj}

We now given an example (among many that we have found) which provide evidence for Conjecture \ref{conj:lin-factors}.

\begin{eg}
    Recall the chamber $\gamma$ of $G_6$ given in Example \ref{eg:main-theorem}. This ($6$-dimensional) chamber is semi-external as its intersection with the facet $f$ of $\pos(G_6)$ with minimal inner normal $\BFG{\iota} = (0, 0, 0, 0, 0, 1)$ has dimension $5$. There are exactly five columns of $G_6$ not on $f$ (these columns correspond to the vertex pairs containing $v_6$). 
    In this case we see that on $\semigroup{G_6}$ the polynomial $p_{G_6}^{\gamma}$ does indeed have the linear factors
    \begin{equation}
        d_6 + 1, \dots, d_6 + 4.
    \end{equation}
\end{eg}

\begin{eg}
    Consider the matrix 
    \begin{equation*}
    A^{2,3} := \begin{bmatrix}
1 & 0 & 0 & 1 & 0 & 0 & 0 & 1 & 1 & 1 & 1 \\
0 & 1 & 0 & 0 & 1 & 1 & 1 & 1 & 1 & 2 & 2 \\
0 & 0 & 1 & 1 & 1 & 1 & 2 & 1 & 2 & 2 & 3 
    \end{bmatrix}.
    \end{equation*}
    The chamber $\gamma$ with minimal ray generators 
    \begin{equation*}
        \begin{bmatrix}
            0 \\ 
            0 \\
            1
        \end{bmatrix},
        \begin{bmatrix}
            0 \\ 
            1 \\
            2
        \end{bmatrix},
        \begin{bmatrix}
            1 \\ 
            2 \\
            5
        \end{bmatrix}
    \end{equation*}
    is semi-external as it $2$-dimensionally intersects the facet $f$ of $\pos(A^{2,3})$ with minimal inner normal 
    \begin{equation*}
         \BFG{\iota} := (1,0,0).
    \end{equation*}
    The dot product condition of Conjecture \ref{conj:lin-factors} is satisfied in this case since $\iota \cdot \BF{c} =1$ for each column $\BF{c}$ of $A^{2,3}$ which is not on $f$ (these are columns $1,4,8,9,10,11$, so $k=6$). Additionally, $p_A^{\gamma}(\BF{b})$ is polynomial, as we have verified with \emph{Barvinok}, by computing that 
    \begin{equation*}
        p_A^{\gamma}(\BF{b}) = \frac{(b_1 + 1)(b_1 + 2)(b_1 + 3) (b_1 + 4)(b_1 + 5)(-b_1 + b_2 + 2)(9b_1^2 - 14b_1b_2 + 7b_2^2 - 16b_1 + 28b_2 + 21)}{5040}.
    \end{equation*}
    Therefore, the conditions of Conjecture \ref{conj:lin-factors} are satisfied, and we see that the linear factors $b_1 + 1, \dots, b_1 + 5$ do indeed appear in $p_A^{\gamma}(\BF{b})$. Finally, we remark that this matrix $A^{2,3}$ is in the aforementioned family of matrices $A^{m,n}$ related to the Kronecker coefficients, and is studied in \cite{MiTr22}. 
\end{eg}

We remark that Theorem \ref{theo:external-chamber-binom} gives a characterization (up to saturation) of when $p_A^{\gamma}$ is polynomial in the case that $\gamma$ is an external chamber. It is natural to ask if such a characterization exists for semi-external chambers.

We conclude this section by remarking that the result of Baldoni and Vergne (Theorem \ref{theo:unimodular} in this article) can be used to simplify one of our previous proofs. 

\begin{rem}
Corollary \ref{cor:unimodular} can also be derived from the linear factor result of Baldoni and Vergne (our Theorem \ref{theo:lin-factors}), since in this case Theorem \ref{theo:lin-factors} predicts the linear factors 
\begin{equation*}
    \BFG{\iota}~\cdot~\BF{b}~+~1, \dots, \BFG{\iota}~\cdot~\BF{b}~+~n-d-1.
\end{equation*}
The constant can then be computed by observing that $p_A^{\gamma}(\BF{0}) = 1$.
\end{rem}

\section{Future work} \label{sec:future}

\subsection{Beyond vector partition functions}
There are many functions which are piecewise quasi-polynomial and whose pieces are chambers of fans, but which are not (directly) vector partition functions. However, it seems that a more generalized version of the determinantal formula given in Theorem \ref{theo:ehrhart-det} holds in certain settings. 

We describe the general set-up explicitly now. We note that the function we are describing in this setting has been described as \emph{vector partition-like function} by Briand, Rosas, and Orellana~\cite[Definition 1]{BrOrRo09}, so we shall also use this nomenclature. A function $F$ is \emph{vector partition-like} if it is a piecewise quasi-polynomial whose pieces are chambers of a fan~$\Gamma$. We say that a ray $R$ (1-dimensional cone) of $\Gamma$ is an \emph{$F$-external ray} if $h(tR)$ is a degree $0$ quasi-polynomial in $t$. If this is not the case, we say that $R$ is \emph{$F$-internal}. Likewise, we call any generator of an $F$-external ray (resp. $F$-interal ray), an \emph{$F$-external ray generator} (resp. \emph{$F$-interal ray generator}).
Moreover, we say that a chamber $\gamma$ of $\Gamma$ is an \emph{$F$-external chamber} if all but one ray of $\gamma$ is $F$-external.
Although it may now be clear what the generalized determinantal formula should look like, we record it here for the sake of clarity. 

Let $\gamma$ be an $F$-external chamber with $F$-external ray generators $\BF{v}_1, \dots, \BF{v}_{d-1}$, and $F$-internal ray generator $\BF{v}_{d}$. The determinantal formula in this case would yield:

\begin{equation*}
    F^{\gamma}(\BF{b}) = h\left(\frac{\det(\BF{v}_1, \dots, \BF{v}_{d-1}, \BF{b})}{\det(\BF{v}_1, \BF{v}_2, \dots, \BF{v}_{d})}\right)
\end{equation*}
where $h(t) := F(t\BF{v}_d).$

Of course, this will not hold for all vector partition-like functions $F$. However, we now describe two cases in which this determinantal formula does seem to hold. For these examples it will be useful to introduce the notation $\mathcal{P}_k$ to denote the set of partitions of length at most $k$. Finally, we remark that both of the vector partition function-like functions that we consider in these examples can be written via formulations involving vector partition functions. For these formulations, we direct the reader to \cite{Ra04} and \cite{BrOrRo09}.

\paragraph{\large{Littlewood-Richardson coefficients:}} The \emph{Littlewood-Richardson} (LR) coefficients $c_{\lambda, \mu}^{\nu}$ are the structure constants appearing from the ordinary multiplication of Schur functions. Due to Rassart \cite{Ra04}, it is known that the \emph{Littlewood-Richardson} function $\Phi_k : \mathcal{P}_k \times \mathcal{P}_k \times \mathcal{P}_k \to \mathbb{Z}_{\geq 0}$ described by
\begin{equation*}
    \Phi_k(\lambda, \mu, \nu) = c_{\lambda, \mu}^{\nu}
\end{equation*}
is a vector partition-like function. Moreover, $\Phi_k$ is piecewise polynomial (each of the quasi-polynomials have period $1$ and are thus polynomials). 
We call the fan whose chambers are the domains of polynomiality of $\Phi_k$, $\lrfan{k}.$

The LR function $\Phi_3$ is described explicitly in \cite{Ra04}. In this case, each of the 18 chambers is $\Phi_3$-external and it is shown in \cite{BrRoTr23} that the determinantal formula holds in this case. However, the proof in this case heavily depends on the fact that $\Phi_3$ is a degree $1$ polynomial. 

The LR function $\Phi_4$ has not been computed explicitly (only partially), however the corresponding fan $\lrfan{4}$ has been computed in \cite{BrRoTr20}. The function $\Phi_4$ is a degree $3$ piecewise polynomial, and the fan $\lrfan{4}$ has $67769$ chambers and $515$ rays. We now give an example below, illustrating that the determinantal formula also applies to $\Phi_4$.

\begin{eg} \label{eg:lr-binomial}
Consider the facet $f$ of $\lrcone{4}$ defined by the equation
\begin{equation*}
    \lambda_1 + \lambda_3 + \mu_1 + \mu_2 = \nu_1 + \nu_3
\end{equation*}
and thus having minimal inner normal 
\begin{equation*}
    \BFG{\iota} := (1, 0, 1, 0, 1, 1, 0, 0, -1, 0, -1, 0).
\end{equation*}
The facet $f$ is contained in the chamber $\kappa_{67709}$ of $\mathcal{L}\mathcal{R}_4$ with minimal ray generators
\begin{alignat*}{8}
\BF{v}_1 &:= \left(0, 0, 0, 0, 1, 0, 0, 0, 1, 0, 0\right) & \quad &
\BF{v}_2 &:= \left(0, 0, 0, 0, 1, 1, 1, 0, 1, 1, 1\right) & \quad &
\BF{v}_3 &:= \left(0, 0, 0, 0, 1, 1, 1, 1, 1, 1, 1\right) \\
\BF{v}_4 &:= \left(1, 0, 0, 0, 0, 0, 0, 0, 1, 0, 0\right) & \quad &
\BF{v}_5 &:= \left(1, 1, 0, 0, 0, 0, 0, 0, 1, 1, 0\right) & \quad &
\BF{v}_6 &:= \left(1, 1, 0, 0, 1, 0, 0, 0, 1, 1, 1\right) \\
\BF{v}_7 &:= \left(1, 1, 0, 0, 1, 1, 0, 0, 2, 1, 1\right) & \quad &
\BF{v}_8 &:= \left(1, 1, 0, 0, 1, 1, 1, 0, 2, 1, 1\right) & \quad &
\BF{v}_9 &:= \left(1, 1, 1, 0, 0, 0, 0, 0, 1, 1, 1\right) \\
\BF{v}_{10} &:= \left(1, 1, 1, 1, 0, 0, 0, 0, 1, 1, 1\right) & \quad &
\BF{v}_{11} &:= \left(4, 3, 1, 0, 3, 2, 1, 0, 6, 4, 3\right) & & &
\end{alignat*}
of which each is $\Phi_4$-external ray generator with the exception of $\left(4, 3, 1, 0, 3, 2, 1, 0, 6, 4, 3\right)$, and so $\kappa_{67709}$ is a $\Phi_4$-external chamber. We have computed by interpolation that the polynomial $\Phi_4^{\kappa_{67709}}$ is 
\begin{equation*}
\Phi_4^{\kappa_{67709}}(\lambda, \mu, \nu) = \binom{\lambda_1 + \lambda_3 + \mu_1 + \mu_2 - \nu_1 - \nu_3 + 3}{3}.
\end{equation*}
Additionally, we have computed (also by interpolation) that $h(t) := \Phi_4^{\kappa_{67709}}(t\BF{v}_{11}) = \binom{t+3}{3}$. Finally, we checked that 
\begin{equation*}
   \frac{\det(\BF{v}_1, \dots, \BF{v}_{10}, \BF{b})}{\det(\BF{v}_1, \BF{v}_2, \dots, \BF{v}_{11})}
\end{equation*}
is indeed $\lambda_1 + \lambda_3 + \mu_1 + \mu_2 - \nu_1 - \nu_3$ where $\BF{b} = (\lambda | \mu | \nu)$ is the vector whose coordinates are the parts of $\lambda, \mu, \nu$.
\end{eg}

In all, there are $21$ $\Phi_4$-external chambers of $\lrfan{4}$. We have verified that the determinantal formula does indeed apply in each of these cases.

\paragraph{\large{Kronecker coefficients:}} \
The \emph{Kronecker coefficients} $g_{\lambda, \mu, \nu}$ are the structure constants in the decomposition of a tensor product of irreducible representations of the symmetric group into irreducible representations:
\begin{equation*}
	V_{\mu} \otimes V_{\nu} = \bigoplus\limits_{\lambda} g_{\lambda, \mu, \nu} V_{\lambda}.
\end{equation*}
Due to a result of Meinrenken and Sjamaar \cite{MeSj99}, it is known that the \emph{Kronecker function} $G_{m,n} : \mathcal{P}_{mn} \otimes \mathcal{P}_m \otimes \mathcal{P}_{n} \to \mathbb{Z}_{\geq 0}$ defined by $G_{m,n}(\lambda, \mu, \nu) = g_{\lambda, \mu, \nu}$ is a vector partition-function. 

The Kronecker function $G_{2,2}$ (i.e the $m=n=2$ case) has been described explicitly by Briand, Rosas, and Orellana \cite{BrOrRo09}. In this case there are $74$ chambers, each of dimension $5$, and the degree of the piecewise quasi-polynomial $G_{2,2}$ is $2.$ The function $G_{2,2}$ is described implicitly, as a simpler function auxiliary function (which we shall call $G^*_{2,2}$) is described, from which $G_{2,2}$ may be retrieved (see \cite[Remark 2]{BrOrRo09}). 

The function  $G^*_{2,2} : \mathbb{Z}^5 \to \mathbb{Z}_{\geq 0}$ (the name of this function is our own) is defined via $G_{2,2}$ by the relation
\begin{equation}
    G^*_{2,2}(n, r, s, g_1, g_2) := G_{2,2}(n - g_1 - g_2, g_1, g_2, n-r, r, n - s, s).
\end{equation}
We now show that the determinantal formula does indeed hold for $G^*_{2,2}$, and thus also for $G_{2,2}$.

\begin{eg}
Consider the chamber $\gamma_{1}$ (numbering as given in the Sage notebook \emph{The chamber complex and quasipolynomial formulas for the 3,2,2-Kronecker coefficients}, Briand 2022) of $\gamma_{62}$ of $G^*_{2,2}$ with ray generators
\begin{align*}
\BF{v}_1 := \left(3, 1, 1, 1, 1\right), \ 
\BF{v}_2 := \left(1, 0, 0, 0, 0\right), \
\BF{v}_3 := \left(2,1,0,1,0\right), \
\BF{v}_4 := \left(2,0,1,1,0\right), \
\BF{v}_5 := \left(6,2,2,2,1\right).
\end{align*}

Each of $\BF{v}_1, \dots, \BF{v}_4$ are $G^*_{2,2}$-external ray generators since $G^*_{2,2}(t\BF{v_i}) = 1$ for each $i=1,2,3,4$. Also, $\BF{v_5}$ is not a $G^*_{2,2}$-external ray generator since 
\begin{align*}
    f &:= G^*_{2,2}(t\BF{v_5}) \\
    &= \begin{cases}
        (t+1)^2 & \text{ if } t \equiv 0 \mod 2 \\
        (t+2)^2/4 & \text{ if } t \equiv 1 \mod 2.
    \end{cases} 
\end{align*}
Therefore, $\gamma_{62}$ is $G^*_{2,2}$-external, and so our determinantal formula predicts that
\begin{align*}
    {G^*_{2,2}}^{\gamma_{62}}(n, r, s, g_1, g_2) &= h\left(\frac{\det(\BF{v}_1, \dots, \BF{v}_4, \BF{b})}{\det(\BF{v}_1, \dots, \BF{v}_4, \BF{v}_5)}\right) \\
    &= \begin{cases}
        (r + s - g_1 - g_2 +1)^2 & \text{ if } r + s - g_1 - g_2 \equiv 0 \mod 2\\
        \frac{(r + s - g_1 - g_2 +1)(r + s - g_1 - g_2 +3)}{4} & \text{ if } r + s - g_1 - g_2 \equiv 1 \mod 2
    \end{cases}
\end{align*}
where $\BF{b} = (n, r, s, g_1, g_2)$. This is indeed confirmed by the computations made in the aforementioned notebook. 
\end{eg}

In all, the function $G^*_{2,2}$ has $7$ external chambers ($\gamma_1, \gamma_{59}, \gamma_{61, }\gamma_{62}, \gamma_{67}, \gamma_{69}, \gamma_{70}, \gamma_{72}$). The determinantal formula holds in each case. 

Interestingly, in some of these cases, the quasi-polynomial is polynomial, and (reflecting again our work in the vector partition case) is given by a negative binomial coefficient. We reproduce one of these cases below. 

\begin{eg}
Consider the chamber $\gamma_{62}$ of $G^*_{2,2}$ with ray generators
\begin{align*}
\BF{v}_1 := \left(3, 1, 1, 1, 1\right), \ 
\BF{v}_2 := \left(1, 0, 0, 0, 0\right), \
\BF{v}_3 :=\left(4, 1, 2, 1, 1\right), \
\BF{v}_4 := \left(2, 0, 1, 1, 0\right), \
\BF{v}_5 := \left(10, 3, 4, 3, 2\right).
\end{align*}
Each of $\BF{v}_1, \dots, \BF{v}_4$ are $G^*_{2,2}$-external ray generators since $G^*_{2,2}(t\BF{v_i}) = 1$ for each $i=1,2,3,4$. Also, $\BF{v_5}$ is not a $G^*_{2,2}$-external ray generator since $h := G^*_{2,2}(t\BF{v_5}) = \binom{t + 2}{2}$. Therefore, $\gamma_{62}$ is $G^*_{2,2}$-external, and so our determinantal formula predicts that
\begin{align*}
    {G^*_{2,2}}^{\gamma_{62}}(n, r, s, g_1, g_2) &= h\left(\frac{\det(\BF{v}_1, \dots, \BF{v}_4, \BF{b})}{\det(\BF{v}_1, \dots, \BF{v}_4, \BF{v}_5)}\right) \\
    &= \frac{(r - g_2 + 2)(r - g_2 + 1)}{2}
\end{align*}
where $\BF{b} = (n, r, s, g_1, g_2)$. Again, we have confirmed the computations by checking against Briand's \emph{Sagemath} notebook. 
\end{eg}

\paragraph{\large{Summary:}}

The cases shown in this section constitute only isolated examples in which the determinantal formula seems to hold in a more general setting. There are several natural questions which one may ask.

\begin{que}
    What are the necessary conditions on a vector partition-like function $F$ so that the determinantal formula holds?
\end{que}

\begin{que}
    Is there a generalization of Theorem \ref{theo:external-chamber-binom} in this setting? In other words, is there a characterization (up to lattice) of when $F^{\gamma}$ is polynomial for an external chamber $\gamma$? In this case, is $F^{\gamma}$ given by a negative binomial coefficient?
\end{que}

\begin{que}
    Is there a generalization of Theorem \ref{theo:external-col-variables} that holds in this setting (i.e are we able to reduce dimension for certain chambers in a more general setting than vector partition functions)?
\end{que}

\subsection{Linear factors}

The main result that we would like to obtain here is a proof of Conjecture \ref{conj:lin-factors} (or a similar result, generalizing that of Baldoni and Vergne for the unimodular case). We formulate this below as a question.

\begin{que}
    Let $f$ be a facet of $\pos(A)$ containing $k$ columns of $A$ for some $d-1 \leq k \leq n$, and let $\iota$ be the inner facet normal of $f$. Let $\gamma$ be a semi-external chamber of $A$ with a facet contained in $f$. What conditions must be placed on $A$ and $\gamma$ such that 
    \begin{enumerate}
        \item $p_A^{\gamma}$ is polynomial, and
        \item $\iota \cdot \BF{y} + 1, \dots, \iota \cdot \BF{y} + n - k - 1$ each appear as linear factors of $p_A^{\gamma}(\BF{y})$?
    \end{enumerate}
\end{que}

 In the case that Conjecture \ref{conj:lin-factors} holds, it may be interesting to see if one can compute linear factors for the matrices $G_m$ of Section \ref{sec:multigraph} (i.e associated to the enumeration of multigraphs). As it is not difficult to compute the facets of $\pos(G_m)$ for all $m$, one could then compute the associated linear factors and see what combinatorial information may be derived from such an analysis.

Finally, we remark that linear factors have been observed in the vector partition-like function setting as well. For example, for LR coefficients \cite{KiToTo07}, as well as for Kronecker coefficients\footnote{From Baldoni, Vergne, Walter \cite{BaVeWa17}: ``We summarize our results in Table 3. We find that, remarkably, the symbolic function on $\mathfrak{c}_{v_{F_I}}$
is polynomial, instead of merely quasipolynomial
(first row). It is a striking fact that this polynomial function
is divisible by 7 linear factors with constant values 1, 2, 3, 4, 5, 6, 7 on the
face $F_I$.''}. It would thus be interesting to explore the appearance of linear factors in the vector partition-like function case as well.

\subsection{Persistent chambers}
Recall that in Section \ref{sec:multigraph} we are able to produce an external chamber $\gamma^{(m)}$ for each number of vertices $m$. This sequence of external chambers $\gamma^{(m)}$ form a family whose inequalities are each determined by a single map depending only on $m$ (i.e given only $m$ we can compute the inequalities defining $\gamma^{(m)}$). Let us (somewhat informally) call such a sequence of chambers $(\gamma^{(m)})$ \emph{persistent}. An interesting avenue of potential research would be to study persistent chambers related to other combinatorial classes that can be viewed through the vector partition function lens. In the multigraph enumeration proof, the vector partition function perspective yields the correct inequalities to look at, but the combinatorial proof is simpler than the geometrical one. It would be a very interesting result to find a combinatorial class for which the geometrical proof is also simpler than the combinatorial one. 

\subsection{Bounds on the number of chambers}
The problem of computing the number of chambers of a given matrix $A$ is a surprisingly difficult one. For example, the matrices $K_m$ associated to Kostant's partition function are easy to describe, however, the corresponding number of chambers is only known for the first seven cases.

Given the number of ``walls'' $m$ of the chamber complex (i.e hyperplanes defining codimension $1$ cones of the chamber complex), and the dimension $d$, there are polynomial upper bounds on the number of chambers of $A$ from the theory of hyperplane arrangements.  Namely, it is known that the number of chambers is at most 
\begin{equation}
    \binom{m}{0} + \binom{m}{1} + \dots + \binom{m}{d}
\end{equation}
(see for example \cite[Lemma 3.3]{VeWo08}). 

A consequence of our work in proving Lemma \ref{lem:dim-reduction} (specifically Proposition \ref{prop:chambers-for-smaller}) allows us to bound the number of chambers of $A$ from below. Assume without loss of generality that $\BF{a}_1$ is an external chamber of~$A$. By possibly applying an appropriate linear map $M$, we can assume that $\BF{a}_1 = k_1\BF{e}_1$ for some positive integer $k_1$. Then, we find that 
    \begin{equation*}
        \# \text{ chambers of } A \text{ containing } \BF{a}_1 = \# \text{ chambers of } A_{\hat{1}, \hat{1}}
    \end{equation*}
    and since 
    \begin{equation*}
        \# \text{ chambers of } A \text{ not containing } \BF{a}_1 \geq \# \text{ chambers of } A_{\cdot, \hat{1}}
    \end{equation*}
    we can derive the following inequality
    \begin{equation*}
        \# \text{ chambers of } A \geq \# \text{ chambers of } A_{\hat{1}, \hat{1}} + \text{ chambers of } A_{\cdot, \hat{1}}.
    \end{equation*}
It may be interesting to compute the lower bounds that are obtained by applying this process recursively. 

\section*{Acknowledgements}
I would like to thank Matthias Beck for suggesting that Corollary \ref{cor:unimodular} should be true via the approach of Lemma \ref{lem:coin-exchange}. I would also like to thank Nathan Ilten for pointing out that one need not assume $\gamma$ is simplicial in Theorem \ref{theo:external-col-variables}. Additionally,
I would like to thank Marni Mishna, Nathan Ilten, Alejandro Morales, Emmanuel Briand, and Matthias Beck for numerous edits and suggestions. Moreover, I would like to thank Jes\'us de Loera for providing me with data for Kostant's partition function.
Finally, I would like to thank Mercedes Rosas and Sheila Sundaram (and Marni Mishna once again) -- Example 2 of their article~\cite{MiRoSu21} provided the inspiration for our proof of Theorem \ref{theo:external-col-variables}.

The author benefited from the financial support of the
National Science and Engineering Research Council (NSERC) of Canada, via NSERC Discovery Grant
R611453, and also Universidad de Sevilla via collaboration with the project \emph{Singularidades, geometría algebraica aritmética y teoría de representaciones} P20-01056.

\bibliographystyle{plain}
\bibliography{references}
\end{document}